\documentclass[11pt]{article}

\usepackage{graphicx,epsfig,amssymb,amsbsy,amsmath,amsthm}
\usepackage{color}
\newcommand{\bb}[1]{\left({#1}\right)}					
\newcommand{\sq}[1]{\left[#1\right]}						
\newcommand{\cc}[1]{\left\{#1\right\}}					
\newcommand{\op}[1]{\mathcal{#1}}
\newcommand{\ord}[1]{{\mathcal{O}}\bb{#1}}
\newcommand{\abs}[1]{\left|#1\right|}					

\newtheorem{theorem}{Theorem}
\newtheorem{lemma}[theorem]{Lemma}
\newtheorem{proposition}[theorem]{Proposition}
\newtheorem{corollary}[theorem]{Corollary}
\newtheorem{definition}{Definition}

\newtheorem{remark}{Remark}

\newcommand{\dd}{\delta_0}

\newcommand{\KK}{\mathcal{K}}

\newcommand{\OO}{\mathcal{O}}

\newcommand{\eps}{\varepsilon}

\newcommand{\al}{\alpha}

\newcommand{\fref}[1]{figure~\ref{#1}}

\newcommand{\eref}[1]{(\ref{#1})}
\newcommand{\erefs}[2]{(\ref{#1})-(\ref{#2})}
\newcommand{\sref}[1]{section~\ref{#1}}

\newcommand{\cref}[1]{chapter~\ref{#1}}

\def\eps{\varepsilon}

\begin{document}

\title{A unified approach to explain contrary effects of hysteresis and smoothing in nonsmooth systems}%
\author{
Carles Bonet \& Tere M. Seara\footnote{Departament de Matem\`atiques, Universitat Polit\`ecnica de Catalunya}\\
Enric Fossas\footnote{Institut d'Organitzaci\'o i Control de Sistemes Industrials, Universitat Polit\`ecnica de Catalunya},
Mike R. Jeffrey\footnote{Department of Engineering Mathematics, University of Bristol, UK, email: mike.jeffrey@bristol.ac.uk}}
\date{\today}

\maketitle




\begin{abstract}
Piecewise smooth dynamical systems make use of discontinuities to model switching between regions of smooth evolution.
This introduces an ambiguity in prescribing dynamics at the discontinuity: should it be given by a limiting value on one side or other of the discontinuity,
or a member of some set containing those values? One way to remove the ambiguity is to {\it regularize} the discontinuity,
the most common being either to smooth out the discontinuity, or to introduce a hysteresis between switching in one direction or the other across the discontinuity.
Here we show that the two can in general lead to qualitatively different dynamical outcomes. We then define a higher dimensional model with both smoothing and hysteresis, and study the competing limits in which hysteretic or smoothing effect dominate the behaviour, only the former of which correspond to Filippov's standard `sliding modes'.
\end{abstract}

\section{Introduction}\label{sec:intro}


The existence of solutions to a system of ordinary differential equations is well established if they are sufficiently smooth \cite{lindelof}.
Even at places where the equations are discontinuous, existence of solutions can be proven using
the theory of differential inclusions \cite{f88}.
To explicitly describe those solutions is another problem, however.
The most commonly used formalism is due to Filippov \cite{f88} and its control application by Utkin \cite{u77}.
They essentially approximate chattering to-and-fro across a discontinuity by a steady flow precisely along the discontinuity.
Utkin's method is sometimes misinterpreted as being different to Filippov's, if taken literally (when in fact the intended outcome is the same, see e.g. \cite{u92}).
Whereas Filippov describes a linear (and hence convex) combination of vector fields, $\lambda{\bf f}^++(1-\lambda){\bf f}^-$ for $\lambda\in[0,1]$,
Utkin describes a function ${\bf f}({\bf x},u)$ where $u\in[0,1]$ and ${\bf f}({\bf x},1)\equiv{\bf f}^+$, ${\bf f}({\bf x},0)\equiv{\bf f}^-$.
While Utkin intends this function to be exactly Filippov's linear combination (i.e. $u=\lambda$), the formulation does raise the question: what if the dependence on $\lambda$ or $u$ is nonlinear?
It is well known that nonlinear dependence on the switching quantity can produce different dynamics (see \cite{j13error}),
but the precise conditions under which it does so are a subject of ongoing study.
This distinction is an important one, since the burgeoning theory of discontinuity-induced bifurcations relies heavily on the canonical form
of dynamics due to Filippov, and very little of the established theory applies for nonlinear dependence on $\lambda$ in general.

Important contributions to the theory of these methods include \cite{aizerman12,and59,bc08,fluggelotz,krg03,tsg11},
and while alternatives exist they do not resolve the ambiguity at the discontinuity \cite{hajek1,hermes68,j13error,t07}.
Most authors follow Filippov by convention, particularly in the growing theory of discontinuity-induced singularities and bifurcations.
Much generality is lost from the current theory by ignoring this issue, however, and unnecessarily so,
for the same methods used to study Filippov systems can be extended to the more general systems admitting {\it nonlinear switching}.


We can illustrate the disparity between dynamics subject to linear and nonlinear switching with a simple example proposed by Filippov and Utkin themselves
(given in \cite{f88,u92}).
Consider the planar piecewise-smooth system
\begin{equation}\label{exutkin}
\dot x=0.3+u^3\qquad\dot y=-0.5-u\qquad u={\rm sign}(y)\;.
\end{equation}
In $y\neq0$ the solutions are simply straight trajectories that travel towards $y=0$, called the {\it switching surface}, and hit it in finite time.
Since they cannot then leave $y=0$, the solutions for all later times must satisfy $\dot y=0$, and are said to {\it slide} along the switching surface.
We use this condition to find the value of $u$ on $y=0$.
Filippov's and Utkin's manners of finding these sliding trajectories imply a linear or nonlinear treatment of \eref{exutkin}:
\begin{itemize}
\item nonlinear (Utkin's formulation): the vector field as written above has a continuous dependence on $u$ with $u\in\sq{-1,+1}$,
so simply solve $\dot y=-0.5-u=0$ on $y=0$ to find $u=-0.5$, then taking the expression for $\dot x$ we have $$\dot x=0.3+(-0.5)^3=0.175\;.
$$
\item
linear (Filippov's formulation):
the vector $(\dot x,\dot y)$ jumps between the values $(1.3,-1.5)$ and $(-0.7,0.5)$ across $y=0$, so assume on $y=0$
it is a convex combination $(\dot x,\dot y)=\lambda(1.7,-1.5)+(1-\lambda)(-0.7,0.5)$ with
$\lambda\in\sq{0,1}$, and solve $\dot y=0$ 
to find $\lambda=0.25$, then the convex combination of $\dot x$ values gives $$\dot x=0.25(1.3)+(1-0.25)(-0.7)=-0.2\;.$$
\end{itemize}
Not only are the magnitudes of the two sliding velocities different, but they are in opposite directions.
Along $y=0$, Filippov's approach predicts motion to the left while Utkin's predicts motion to the right!
These are illustrated in \fref{fig:ut0}.
%
\begin{figure}[h!]\centering\includegraphics[width=0.95\textwidth]{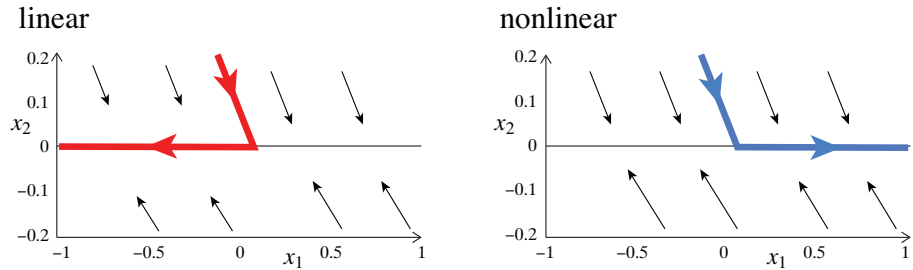}
\vspace{-0.3cm}\caption{\footnotesize\sf The vector field (\ref{exutkin}) with a switching surface $y=0$,
and sliding motion along the surface to the right according to the nonlinear formulation, or to the left according to the linear formulation.
The two figures agree for $y\neq0$, but give opposing solutions on $y=0$.}\label{fig:ut0}\end{figure}

Clearly, to decide between the contrary outcomes we must improve the discontinuous model, but we must be aware of tautologies:
{\it both limiting} solutions can be rigorously proven to be valid under different assumptions, as we will demonstrate.
We clarify the situation by showing that introducing hysteresis in the switch implies that solutions lie close to Filippov's,
while smoothing out the switch implies that solutions lie close to Utkin's. That is, we replace an ideal switch with a boundary layer
which is, in some sense, negative in the Filippov case and positive in the Utkin case. In \sref{sec:u} we unify these contradictory behaviours by proposing a model with both smoothing and hysteresis, achieved by embedding the planar problem in a three dimensional slow-fast system.

Invoking the names of Filippov and Utkin for the two approaches neglects the deeper and more general investigations by these authors,
and their various works are highly recommended for further reading.
In \cite{u92} Utkin suggests that his `equivalent control' method should only be used when $u$ appears linearly in \eref{Ns},
which is precisely the case when it is equivalent to Filippov's method \cite{f88}. However, the two approaches are both powerful and, as we shall see,
both correct in differing scenarios, and it is those scenarios that we seek to better understand here.

Before continuing we make a remark on generality.
The reader will lose nothing by considering $x,y$ and $u$ to be scalars, but all of the following analysis is written in such a way that it applies also
when $x$ is a vector.
For convenience we use terms such as `curve', `surface', etc. as if $x$ were a scalar (e.g. the set $y=0$ is therefore a plane in the space of $x,y,u,$
and the set $u=y=0$ is a line, though more generally these are sets of codimension one and two, respectively).
The analysis can also be extended to multiple discontinuities by letting $u$ be a vector of parameters $u_1, u_2,...$,
each component having a different discontinuity surface $y_1=0,\;y_2=0,...$, however this extension is not trivial and requires further analysis at
points where different discontinuity surfaces intersect, see for example \cite{dieci2011,j13iso}.

The paper is arranged as follows.
In \sref{sec:ns} we review the two canonical methods for solving dynamics at a discontinuity due to Filippov and Utkin,
showing that they can be seen as limits of hysteresis and smoothing respectively.
Our main results are in \sref{sec:u}, where we embed our non-smooth system in a slow-fast smooth system  which, depending on the shape of its critical manifold,
tends either to the linear (Filippov) or nonlinear (Utkin) dynamics.
Some of the lengthier details proving these limits are given in the appendix, after some closing remarks in \sref{sec:conc}.

\section{The discontinuous models}\label{sec:ns}


Let variables $x\in\mathbb R^{n-1}$ and $y\in\mathbb R$ satisfy a differential equation
\begin{equation}\label{Ns}
\begin{array}{rcl}
\dot x &=& f(x,y; u)\\
\dot y &=& g(x,y; u)
\end{array}
\end{equation}
where $f$ and $g$ are smooth functions of $x,y,u,$ and where $u$ is given by
\begin{eqnarray}\label{hyst0}
u={\rm sign}(y)\;.
\end{eqnarray}
The values of the vector field either side of the switch can be written as
\begin{equation}\label{FF}
f^\pm(x,y)=f(x,y; \pm1)\qquad{\rm and}\qquad g^\pm(x,y)=g(x,y; \pm1)\;.
\end{equation}
We will only be interested in the case where the flow is directed towards the switching surface $y=0$ from both sides,
so we restrict to a range of $x$ such that for some $M$,
\begin{equation}\label{g}
\left.\begin{array}{l}g(x,0; +1)<0<g(x,0; -1)\\{\rm and}\qquad\frac{\partial\;}{\partial u}g(x,0; u)<0\end{array}\right\}\quad\mbox{for}\;\;\;x\in [-M,+M] \;.
\end{equation}

While this system is smooth away from $y=0$, equations \erefs{Ns}{hyst0} do not provide a well-defined value for $(f,g)$ on $y=0$.
In a piecewise-smooth dynamics approach to \eref{Ns}-\eref{hyst0}, we attempt to resolve the discontinuity by defining $f(x,y;u)$
and $g(x,y;u)$ in such a way that the system:
\begin{enumerate}
\item coincides with \eref{Ns}-\eref{hyst0} for $y\neq0$,
\item extends $f$ and $g$ to be well-defined for all $(x,y)$.
\end{enumerate}

\subsection{Filippov and Utkin's conventions}\label{sec:fu}

Let us begin by paraphrasing the classic approaches of Filippov's sliding and Utkin's equivalent control, or more correctly, of linear and nonlinear sliding.
Define a solution of \eref{Ns}-\eref{hyst0} that travels along the switching surface $\Sigma =\{(x,y)\in\mathbb R:\;y=0\}$ for an interval of time as follows:

\begin{definition}\label{def:fil}
Filippov's {\rm sliding dynamics} along the discontinuity $y=0$ is given by
\begin{equation}\label{fil}
\left.\begin{array}{rll}
\dot x&=&\lambda f^+(x,0)+(1-\lambda)f^-(x,0)\\
0&=&\lambda g^+(x,0)+(1-\lambda)g^-(x,0)
\end{array}\right\}
\end{equation}
if there exist solutions such that $\lambda\in\sq{0,1}$.
\end{definition}

\begin{definition}\label{def:utkin}
Utkin's equivalent control along the discontinuity $y=0$ is given by
\begin{equation}\label{avcon}
\left.\begin{array}{rll}
\dot x&=&f(x,0;u)\\
0&=&g(x,0;u)\end{array}\right\}
\end{equation}
if there exist solutions such that $u\in\sq{-1,+1}$.
\end{definition}
While Definition \ref{def:fil} permits only linear dependence on the switching quantity (here $\lambda$), Definition \ref{def:utkin} permits nonlinear dependence on the switching quantity (here $u$).

\bigskip
In either case, for a trajectory moving along $y=0$ the component normal to the
{switching surface} must be zero (hence $\dot y=0$),
which gives the algebraic constraint in the second line of each definition.
For \eref{fil} we can solve to find
\begin{equation}
\lambda=\Lambda(x):=\frac{g^-(x,0)}{g^-(x,0)-g^+(x,0)}\qquad{\rm on}\;\;y=0\;,
\end{equation}
which lies in the range $\sq{0,1}$ if $g^+$ and $g^-$ have opposite signs, as given by \eref{g}.
The velocity along the switching surface $y=0$ is then
\begin{eqnarray}
\dot x&=&f_F(x):= f^-(x,0)+\cc{f^+(x,0)-f^-(x,0)}\Lambda\bb{x} \nonumber\\
&=&\frac{f^+g^--f^-g^+}{g^--g^+}(x,0)\;. \label{xfilippov}
\end{eqnarray}

In \eref{avcon} we assume instead that the vector field at the switching surface jumps between $(f^+,g^+)$ and $(f^-,g^-)$ in such a way that the functional forms
$f=f(x,y;u)$ and $g=g(x,y;u)$ remain valid on $y=0$.
We then seek the value of $u\in\sq{-1,+1}$ that ensures a trajectory moves along $y=0$ (and therefore, again, $\dot y=0$), given by the second line of \eref{avcon}.
On a region where $\partial g\bb{x,0;u}/\partial u\neq0$ we can solve this condition to find
\begin{equation}\label{uutkin}
u=U(x), \ \mbox{ such that } \ g(x,0;U(x))=0, \ \forall x \in [-M,M], \ {\rm on}\;\;y=0\;,
\end{equation}
which has a solution in the range $\sq{-1,+1}$ by \eref{g}.
The velocity along the switching surface $y=0$ is then
\begin{equation}\label{xutkin}
\dot x=f_U(x):=f\bb{x,0;U(x)}\;.
\end{equation}

The two systems \eref{fil} and \eref{avcon} (equivalently \eqref{xfilippov} and  \eqref{xutkin}) are  equivalent when $f$ and $g$ depend linearly on $u$,
when we can write
\begin{equation}\label{lin}
\begin{array}{rcl}
f(x,y;u)=a(x,y)+b(x,y)u \;,\\
g(x,y;u)=c(x,y)+d(x,y)u\;,
\end{array}
\end{equation}
with
\begin{eqnarray*}
a&=&\bb{f^-+f^+}/2\;,\qquad b=\bb{f^+-f^-}/2\;,\\
c&=&\bb{g^-+g^+}/2\;,\qquad d=\bb{g^+-g^-}/2\;,
\end{eqnarray*}
where by \eqref{FF}, $g^\pm(x,y)=g(x,y;\pm1)$ and $f^\pm(x,y)=f(x,y;\pm1)$.
Computing $U(x)$ in this case using equation \eqref{uutkin}, we obtain $U(x)=-\frac{c(x,0)}{d(x,0)}$, and the vector field \eqref{xutkin}
gives the same equations as \eqref{xfilippov}.

\bigskip

When $f$ or $g$ depend nonlinearly on $u$, as we saw in example \eqref{exutkin}, the Filippov and Utkin approaches are distinct,
but in the next section we will show that both approaches can be proven to constitute suitable approximations of the dynamics of system \eref{Ns}.
The distinction turns out to be a practical one:
introducing hysteresis in the switch implies that solutions lie close to Filippov's solution $x_F(t)$ of \eqref{xfilippov},
while smoothing out the switch implies solutions lie close to Utkin's solution $x_U(t)$ of \eqref{xutkin}.
If a model is both smooth in $(x,y,u)$ {\it and} can exhibit hysteresis
(which is the likely situation in many physical systems),
then it is unclear which method to apply (see the example in the introduction).


\subsection{The limit of hysteretic and smoothing regularizations}\label{sec:tautology}

Building on previous works (e.g. \cite{slot91,f88,u92}) let us consider two different models for regularizing a switch,
expressible as perturbations of the nonsmooth system \eref{Ns}.
One model introduces hysteresis in the switch over a distance $|y|<\alpha$,
the other smooths out the discontinuity over a boundary layer $|y|<\alpha$, where $\alpha$ is small in both cases.


%

To introduce hysteresis we consider \eref{Ns} but introduce a negative boundary layer, that is, an overlap between  the regions where $u=+1$ or $u=-1$,
over a region $|y|\le\alpha$.
That is,
\begin{eqnarray}\label{hyst}
u\in\left\{\begin{array}{lll}+1&\rm if&y>-\alpha\;,\\\sq{-1,+1}&\rm if&|y|\le\alpha\;,\\-1&\rm if&y<+\alpha\;,\end{array}\right.
\end{eqnarray}
and switching occurs such that a trajectory with $u=-1$ will maintain this value until it reaches the surface $y=+\alpha$, then switch to $u=+1$.
A trajectory with $u=+1$ will maintain this value until it reaches the surface $y=-\alpha$, then switch to $u=-1$.
Proceeding in this way, we will obtain the hysteretic solution that we denote by $(x_h(t), y_h(t))$ (see \fref{fig:ut1}).

\begin{theorem}[Linear sliding dynamics from hysteresis] \label{propfil}
Fix $T >0$ and consider the solution $x_F(t)$ of the Filippov System \eqref{xfilippov} in $\Sigma$,
and assume that $|x_F(t)| <M$ for $0\le t\le T$ where $M$ is given in \eqref{g}.
Then there exists $\alpha _0>0$ and a constant $L>0$ such that,
for any $0<\al\le \al_0$, if we
consider the hysteretic solution $(x_h(t), y_h(t))$ with initial condition
$(x_h(0), y_h(0))=(x_0, \mp\al)=(x_F(0), \mp\al)$, then $x_h$ satisfies:
\begin{equation}\label{uta}
|x_h(t)-x_F(t)| \le L \al \quad 0\le t\le T\;.
\end{equation}
\end{theorem}
\begin{proof}
In Appendix \ref{sec:hproof}.
\end{proof}

Now
we consider again \eref{Ns}, but replace the definition \eref{hyst0} of $u$ with a smooth sigmoid function, such as $u=\phi(y/\al )$ where
\begin{equation}\label{phi}
\phi(w)\in\left\{\begin{array}{lll}{\rm sign}(w)&\rm if&|w|>1\;,\\\sq{-1,+1}&\rm if&|w|\le1\;,\end{array}\right.\;
\end{equation}
with $\phi'(w)>0$ for $|w|<1$.


\begin{theorem}[Nonlinear sliding dynamics from smoothing.] \label{proput}
Fix $T >0$ and consider the solution $x_U(t)$ of the Utkin's equivalent control \eqref{xutkin} in $\Sigma$,
and assume that $|x_U(t)| <M$ for $0\le t\le T$ where $M$ is given in \eqref{g}.
Then there exists $\al _0>0$ and a constant $L>0$ such that,
for any $0<\al \le \al _0$, if we
consider the smooth system \eref{Ns} where $u=\phi(y/\al )$,
a solution of this system
 $(x(t), y(t))$ with initial condition
$(x(0), y(0))=(x_0, y_0)=(x_U(0), y_0)$, $y_0 \in [-\al ,\al ]$, satisfies
\begin{equation}
|x(t)-x_U(t)| \le L \al  \quad 0\le t\le T\;.
\end{equation}
\end{theorem}

\begin{proof}
In Appendix \ref{sec:sproof}.
\end{proof}

The two theorems are illustrated in \fref{fig:ut1}, where (\ref{exutkin}) is simulated using hysteresis or smoothing to determine the sliding dynamics.
\begin{figure}[h!]\centering\includegraphics[width=0.95\textwidth]{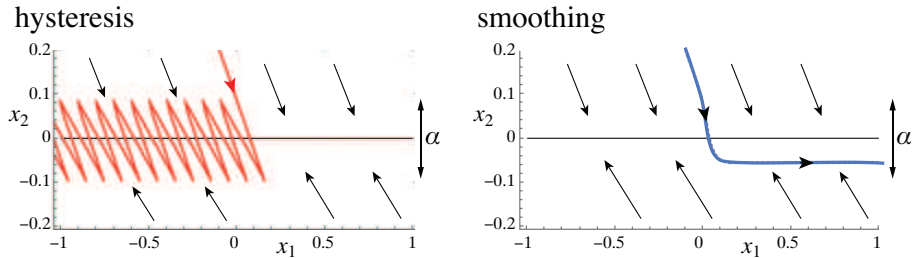}
\vspace{-0.3cm}\caption{\footnotesize\sf Sliding dynamics simulated using hysteresis or smoothing, applied to the example \eref{exutkin}.
The two figures agree outside the regularization strip ($|y|>\alpha$),
but give opposing solutions inside. As $\alpha\rightarrow 0$ these tend to \fref{fig:ut0}.}\label{fig:ut1}\end{figure}

Hence the tautology that is insufficiently acknowledged in the literature on nonsmooth systems:
it seems that in this problem, 
forming more rigorous models only serves to reinforce the case for each method from a different point of view, without clarifying the physical situations under which each applies.
To resolve the contradiction we require a single unified model capable of exhibiting both behaviours in different limits.
We define a system with two parameters $\epsilon$ and $\alpha$ that give us control over the smoothness and hysteresis in one model, and we are then able to show that one behaviour or the other applies, but in distinct limits.
To ``smooth'' hysteresis requires that we embed the system in a higher dimension.
The embedded system should have steady states $u={\rm sign}(y)$ to which the system collapses on a timescale $\ord\al$,
and between which the system transitions over a distance $|y|=\ord\alpha$.



\section{Regularization by embedding and singular perturbation}\label{sec:u}

We can express the hysteretic problem formed by \eref{Ns} with \eref{hyst} as a differential-algebraic system
\begin{equation}\label{u0}\begin{array}{rll}
\dot x&=&f\bb{x,y;u}\;,\\
\dot y&=&g\bb{x,y;u}\;,\\
0&=&\Phi\bb{y+\alpha u}-u\;,
\end{array}\end{equation}
where $\al \ge 0$ and $\Phi$ is a set-valued step function defined as
\begin{equation}\label{step}
\Phi(z)\in\left\{\begin{array}{lll}{\rm sign}(z)&{\rm if}&z\neq0\;,\\\sq{-1,+1}\; &{\rm if}&z=0\;.\end{array}\right.
\end{equation}
This embeds the $u$-parameterized problem in variables $(x,y)$, inside a surface $u=\Phi(y+\alpha u)$ in the higher dimensional space of variables $(x,y,u)$.
The surface consists of two half-planes, $u=+1$ for $y+\alpha >0$ and $u=-1$ for $y-\alpha<0$, which are consistent with \eref{Ns}-\eref{hyst0} when $\alpha=0$.
These half-planes are connected by a plane segment on which $u=-y/\alpha$ and $|u|<1$, which is consistent with the condition $u\in[-1,1]$ from \eref{hyst}.
Hysteresis manifests as a relaxation between the half-planes $u=1, \, y\ge -\alpha$ and $u=-1, \, y\le \alpha$.

This suggests considering a singular perturbation of \eref{u0},
\begin{equation}\label{slow}
\begin{array}{rll}
\dot x&=&f\bb{x,y;u}\;,\\
\dot y&=&g\bb{x,y;u}\;,\\
\eps\dot u&=&\phi\bb{\frac{y+\alpha u}\eps}-u\;,
\end{array}\end{equation}
where $\phi$ is a smooth function with the form \eref{phi} and $\eps>0$ is a small parameter.
Because by \eref{phi}
\begin{equation}
\begin{array}{lll}
{\displaystyle\lim_{\eps\rightarrow0}}\phi\bb{\frac{y+\alpha u}\eps}
&\!\!\!\!\in\!\!\!\!&{\displaystyle\lim_{\eps\rightarrow0}}\left\{\begin{array}{lcl}{\rm sign}(y+\alpha u)&\rm if&|y+\alpha u|>\eps\\\sq{-1,+1}&\rm if&|y+\alpha u|\le\eps\end{array}\right\}\medskip\nonumber\\
&\!\!\!\!=\!\!\!\!&\left\{\begin{array}{lll}{\rm sign}(y+\alpha u)&\rm if&|y+\alpha u|>0\\\sq{-1,+1}&\rm if&|y+\alpha u|=0\end{array}\right\}=\Phi(y+\alpha u),
\end{array}
\end{equation}
for $\eps=0$ the system \eref{slow} is formally equivalent to the system \eref{u0}, and hence to the system \eref{Ns} with \eref{hyst},
and moreover is formally equivalent to the system \eref{Ns}-\eref{hyst0} in the limit $\alpha=0$.
A proper justification of these statements if given in the following sections.

We have two timescales in \eref{slow}, a slow scale $t$ and a fast scale $t/\eps$ assuming $0\le\eps\ll1$.
The idea is that \eref{slow} is a regularization of \eref{Ns}-\eref{hyst0}, meaning it forms a well-defined problem everywhere including at the
discontinuity and formally agrees with \eref{Ns}-\eref{hyst0} for $y\neq0$ in the limit $\alpha,\eps\rightarrow0$.
This is achieved here by embedding the $(x,y)$ problem with a parameter $u$, in the higher dimensional space $(x,y,u)$,
where $u$ is now a fast variable that relaxes quickly to $u=\pm1$.


We will see in the following sections that
the manifold $u=\phi\bb{\frac{y+\alpha u}\eps}$ takes different shapes for $\alpha$ positive or negative, shown in \fref{fig:embedeps}.
The main results of this paper are Theorems \ref{thm:filalpha} and \ref{thm:utkinalpha} in the next section, which prove that the dynamics of \eref{slow} agrees
either with Definition \ref{def:fil} or Definition \ref{def:utkin} depending of the sign of $\alpha$, for certain parameter restrictions and
up to certain errors which we will derive.

\begin{figure}[h!]\centering\includegraphics[width=0.95\textwidth]{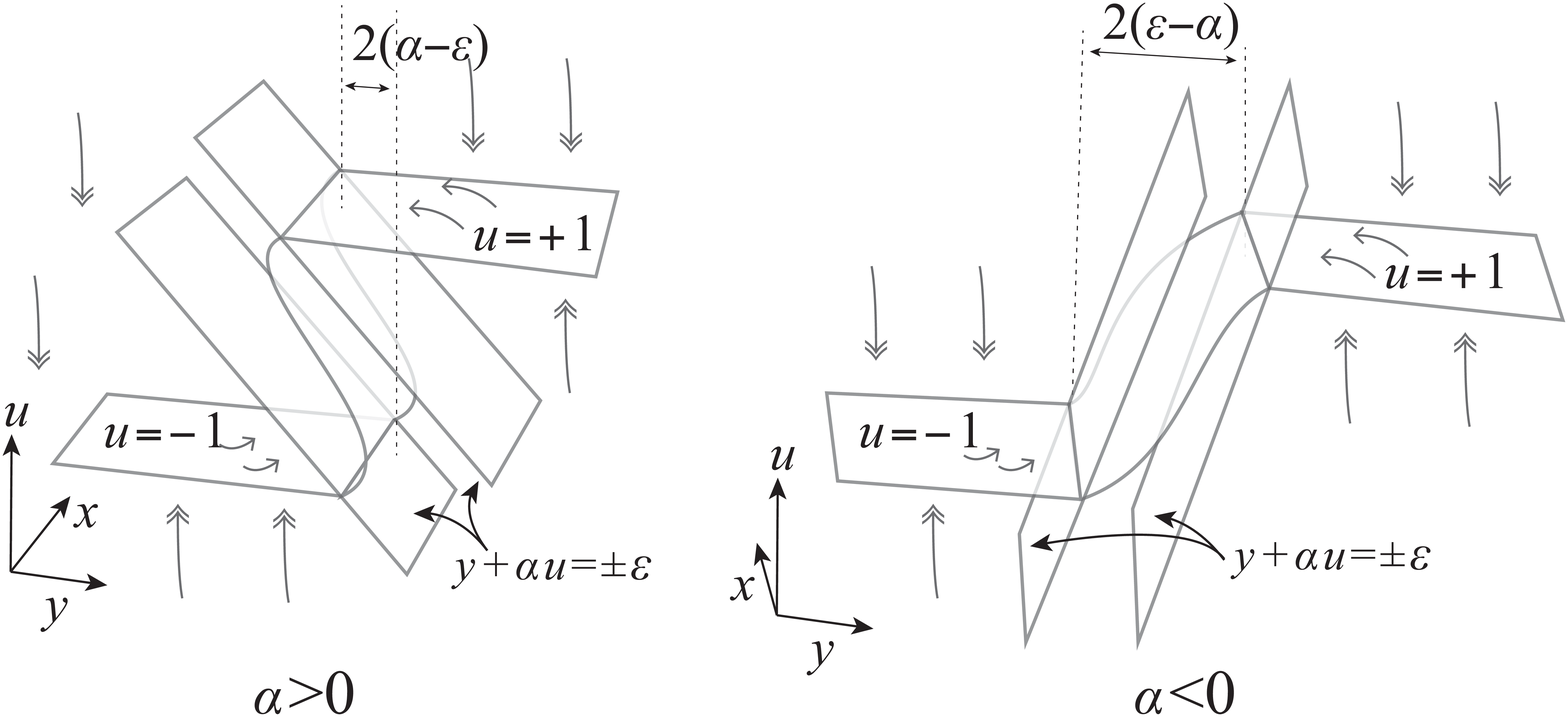}
\vspace{-0.3cm}\caption{\footnotesize\sf A schematic showing key features of the system \eref{slow}.
The surface shown is $u=\phi\bb{\frac{y+\alpha u}\eps}$.}\label{fig:embedeps}\end{figure}

\subsection{Preparatory steps for the theorems}\label{sec:prep}

To properly understand these behaviours for $\eps$ and $\alpha$ small but non-vanishing, let us take a closer look at the multiple
timescale dynamics of the model \eref{slow} from the view of singular perturbation theory.

The ratio of small quantities
\begin{equation}\label{kappa}
\kappa\equiv\eps/\alpha\;,
\end{equation}
will feature in the singular perturbation analysis,
and we assume
\begin{equation}
0<\eps\ll|\alpha|\ll1
\end{equation}
which implies $0<|\kappa|\ll1$.
This is a natural assumption because the relaxation is faster than the switching (which models a ``fast change'' in $u$).


The following theorem relates the solutions of \eref{slow} to those of \eref{xfilippov} of $\alpha>0$.

\begin{theorem}\label{thm:filalpha}
Fix $T >0$, consider $x_F(t)$ the solution of the Filippov System \eqref{xfilippov} in $\Sigma$,
and assume that $|x_F(t)| <M$ for $0\le t\le T$ where $M$ is given in \eqref{g}.
Then there exist constants $C>0$, $L>0$,
$\alpha _0>0$ such that, for any $0<\al\le \alpha _0$,
if we take $0<\kappa< \frac{1}{4}$ and $\dd$ satisfying
$$
2\mathrm{e}^{-\frac{1}{2\kappa\, C}}<\dd \le \kappa\alpha _0,
$$
then the solution $(x(t),y(t),u(t))$ of the system \eref{slow} with $(x(0),y(0),u(0))=(x_0,y_0,u_0)$ such that $|x_0|<M$, $|y_0|<\alpha$ and
$||u_0|-1|<\dd$,
satisfies for all $t\in [0,T]$,
$$
|x(t)-x_F(t)|< L(\kappa+ \frac{ \dd} {\kappa}+ \kappa \abs{\log{\frac{\dd}{2}}}+\al ), \quad |y(t)|<\alpha\;.
$$
\end{theorem}

Taking $\kappa = \al$ and $\dd= \alpha ^2$ one has the following:
\begin{corollary}\label{cor:filalpha}
Fix $T >0$, consider  $x_F(t)$ the solution of the Filippov System \eqref{xfilippov} in $\Sigma$,
and assume that $|x_F(t)| <M$ for $0\le t\le T$ where $M$ is given in \eqref{g}.
Then there exist constants $C>0$, $L>0$, $\al_0 >0$ such that, for $0<\al\le \al _0$ small enough,
the solution $(x(t),y(t),u(t))$ of the system \eref{slow} where $\varepsilon=\alpha ^2$,
with $(x(0),y(0),u(0))=(x_0,y_0,u_0)$ such that $|x_0|<M$, $|y_0|<\alpha$ and
$||u_0|-1|<\al ^2$, satisfies for all $t\in [0,T]$
$$
|x(t)-x_F(t)|< L\al\abs{\log{\frac{\al}{2}}}, \quad |y(t)|<\alpha.
$$
\end{corollary}

The results of Theorem \ref{thm:filalpha} and Corollary \ref{cor:filalpha} jointly with Theorem \ref{propfil} imply that the solutions
of \eref{slow} lie $\al\log\al$ close to those of the hysteretic system \eref{Ns} with \eref{hyst0}.
More precisely, if we take $x_h(t)$ the hysteretic solution given by \eref{uta}, then $x(t)$ in
Theorem \ref{thm:filalpha} satisfies $|y(t)|<\alpha$ and
$$
|x(t)-x_h(t)|<-L \al\log\al \quad \mbox{for all} \  t\in [0,T].
$$

The next theorem relates the solutions of system \eqref{slow} with those of \eqref{xutkin} when $\alpha <0$.

\begin{theorem}\label{thm:utkinalpha}
Take $\al <0$. Fix $T >0$,  $x_U(t)$ consider the solution of the Utkin's equivalent control \eqref{xutkin} in $\Sigma$,
and assume that $|x_U(t)| <M$ for $0\le t\le T$ where $M$ is given in \eqref{g}.
Then there exists $\alpha _0>0$ such that if we take $\dd>0$ and $\kappa<0$ satisfying
$$
0<\dd \le |\kappa|\al_0,
$$
there exists a constant $L>0$ such that,
for any $0<|\al|\le \al_0$,
then the solution $(x(t),y(t),u(t))$ of \eref{slow} with $(x(0),y(0),u(0))=(x_0,y_0,u_0)$ such that
$x_0=x_U(0)$, $|y_0|<|\al|$ and $||u_0|-1|<\dd$,
satisfies for all $t\in\bb{0,T}$:
$$
|x(t)-x_U(t)|<L |\al|, \quad |y(t)|<L |\al|.
$$
\end{theorem}

The results of Theorem \ref{thm:utkinalpha} jointly with Theorem \ref{proput} imply that the solutions
of \eref{slow} lie $\al$ close to those of the smoothing of system \eref{Ns} with \eref{hyst0}.


The proofs of these Theorems are given in the Appendix, as they are in principle rather simple (a matter of showing that solutions are confined either to the neighbourhood of a hysteretic loop or a slow manifold), but in practice are lengthy.
To give an intuitive picture of the dynamics of system \eref{slow} for illustration in \sref{sec:epsns}.

The different orders of approximation between the methods using hysteresis, which is of order $\al|\log \al|$ (from corollary \ref{cor:filalpha}),
or smoothing, which of order $\al$ (from Theorem \ref{thm:utkinalpha}), show their quite different nature.
To have the hysteretic process under control we must ensure that the solution returns sufficently near the manifolds $u=\pm 1$ in each of the $\OO(1 / \alpha)$ hysteresis loops, while in the smoothing process we only need to ensure that solutions reach a certain neighborhood (of the surface $C_0$ described in the next section, or more precisely of the curve $Q$ described in \sref{sec:uproof}) where it is no longer able to escape.



\subsection{A sketch of the $\eps\rightarrow0$ nonsmooth limit}\label{sec:epsns}

Too supplement these results and form a picture of the dynamics, let us explore the system \eref{slow} in the $\eps\rightarrow0$ limit a little more closely, verifying that it fits intuitively with the discontinuous system \eref{Ns} using \eqref{xfilippov} or \eqref{xutkin}.

Letting $\eps\rightarrow0$ in \eref{slow} gives the slow subsystem \eref{u0} on the timescale $t$,
which is discontinuous because $\Phi(z)=\lim_{\eps\rightarrow0}\phi(z/\eps)$ is the step function \eref{step}.
In the space of $(x,y,u)$ this system occupies a surface $\op C$ on which the condition $u=\Phi(y+\alpha u)$ is satisfied.
Expressing this as a graph,
\begin{equation}\label{Cz}
\op C=\cc{(x,y,u)\;:\;u=\mu(y;\alpha)}\;,
\end{equation}
where
\begin{equation}\label{Cmu}
\mu(y;\alpha)=\left\{
\begin{array}{lll}
+1 & \rm if&y\ge -\alpha\;,\\
-y/\alpha&\rm if & |y|\le \alpha\;,\\
-1 & \rm if & y\le +\alpha\;.
\end{array}\right.
\end{equation}
The surface $\op C$ has three branches, two half hyperplanes
\begin{equation}
\begin{array}{rll}
&&\op C_+=\cc{(x,y,u):\;x\in\mathbb R^{n-1},\;y+\alpha\ge0,\;u=+1}\;,\\
&&\op C_-=\cc{(x,y,u):\;x\in\mathbb R^{n-1},\;y-\alpha\le0,\;u=-1}\;,\end{array}
\end{equation}
connected by a hyperplane segment
$$\qquad\qquad\op C_0=\cc{(x,y,u):\;x\in\mathbb R^{n-1},\;y+\alpha u=0,\;u\in[-1,1]},$$
as depicted in \fref{fig:embed0}.
Thus on $\op C=\op C_+\cup\op C_0\cup\op C_-$ the dynamics of \eref{u0} becomes
\begin{equation}\label{slowC}
\begin{array}{rll}
\dot x&=&f\bb{x,y;\mu(y;\alpha)}\;,\\
\dot y&=&g\bb{x,y;\mu(y;\alpha)}\;.
\end{array}
\end{equation}

\begin{figure}[h!]\centering\includegraphics[width=0.95\textwidth]{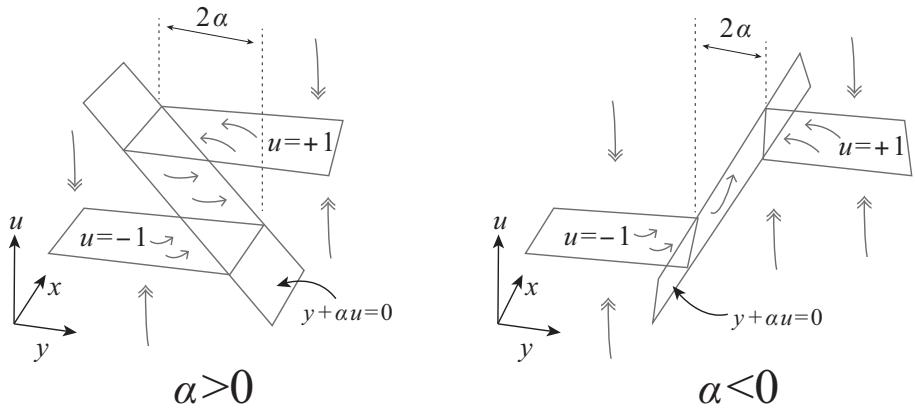}
\vspace{-0.3cm}\caption{\footnotesize\sf Slow dynamics (single arrows) in the surface $u=\Phi(y+\alpha u)$, comprised of subsets of the hyperplanes $u=+1$, $u=-1$,
and $y=-\alpha u$, with fast dynamics (double arrows) outside the surface.}\label{fig:embed0}\end{figure}
%

Denoting the derivative with respect to the fast timescale $t/\eps$ by a prime in \eref{slow} gives
\begin{equation}\label{nsfast}
\begin{array}{rll}x'&=&\eps f\bb{x,y;u}\;,\\
y'&=&\eps g\bb{x,y;u}\;,\\
u'&=&\phi\bb{\frac{y+\alpha u}\eps}-u\;,
\end{array}
\end{equation}
which for $\eps=0$ becomes the one dimensional system
\begin{equation}\label{nsfast0}
\begin{array}{rll}x'&=&0\;,\\
y'&=&0\;,\\
u'&=&\Phi\bb{y+\alpha u}-u\;.
\end{array}
\end{equation}
This induces relaxation towards the surfaces $\op C_\pm$ on the fast timescale, and is a discontinuous one-dimensional system expressible as
$$
u'=\Phi-u\;,\qquad
\Phi\in\left\{\begin{array}{lll}+1&\rm if&y+\alpha u \ge 0\;,\\\sq{-1,1}&\rm if&y+\alpha u=0\;,\\-1&\rm if&y+\alpha u\le0\;,\end{array}\right.
$$
where $y$ is a constant.

The
sets $\op C_\pm$ are therefore half-planes of equilibria of \eref{nsfast0}, where $u'=0$ and $u=\pm1$.
These surfaces are hyperbolically attracting since
$\left.\partial u'/\partial u\right|_{\op C_\pm}=-1$.

The set $\op C_0$ lies on a discontinuity surface of system  \eref{nsfast0} given by $y+\alpha u=0$, so unlike $\op C_\pm$ it is not a set of equilibria.
The value of $u'$ changes sign  across $\op C_0$, but does so discontinuously.
Considering the neighbourhood of $\op C_0$ for which $|y|<\alpha$, for $\alpha>0$ the derivative $u'$ jumps from $-1-u$ to $+1-u$ as $u$ goes from
$u<-y/\alpha$ to $u>-y/\alpha$, so $\op C_0$ is repelling (in finite time), while for $\alpha<0$ the derivative $u'$ jumps from $+1-u$ to $-1-u	$ as $u$ goes
from $u<-y/\alpha$ to $u>-y/\alpha$, so $\op C_0$ is attracting (in finite time).
The following picture of the dynamics then emerges (see \fref{fig:ut2}).


The slow dynamics on $\op C_+$ and $\op C_-$, given by \eref{slowC} with $\mu(y;\alpha)=1$ or $\mu(y;\alpha)=-1$ respectively,
is equivalent to the $u=\pm1$ dynamics of \eref{Ns}.
The surfaces $\op C_\pm$ are invariant except where they meet the switching surface $y+\alpha u=0$,
on two lines $L_1=\cc{(x,y,u):\;y+\alpha=0,u=+1}$ and $L_2=\cc{(x,y,u):\;y-\alpha=0,u=-1}$.
The slow dynamics on $\op C_0$, given by \eqref{slowC} with $u=-y/\alpha$,
is a smooth interpolation between the two systems in \eref{Ns}.

For $\alpha>0$, on the fast timescale, solutions of \eref{nsfast0} are repelled in finite time from the surface $\op C_0$,
and attracted asymptotically towards $\op C_\pm$.
On the line $L_1$ separating $\op C_+$ from $\op C_0$, the flow relaxes
towards the surface $\op C_-$ via the fast system \eref{nsfast0}. On the line $L_2$ separating $\op C_-$
from $\op C_0$, the flow relaxes towards the surface $\op C_+$ again via the fast system \eref{nsfast0}.
Thus the dynamics is consistent with \eref{Ns} using \eref{hyst} for $|y|>\alpha$, and for $|y|<\alpha$ the system jumps
between the slow dynamics on $u=+1$ and $u=-1$ hysteretically.

For $\alpha<0$,
on the fast timescale, solutions of \eref{nsfast0} are attracted asymptotically
towards $\op C_\pm$ and in {\it finite} time towards $\op C_0$.
Hence the surface $\op C=\op C_+\cup\op C_0\cup\op C_-$ is attractive, and, as the dynamics in $\op C_0$ is a regularization of system \eqref{Ns},
it is consistent with  \eqref{avcon} for $|y|<-\alpha$.

The two regimes are simulated in \fref{fig:ut2}.
\begin{figure}[h!]\centering\includegraphics[width=0.95\textwidth]{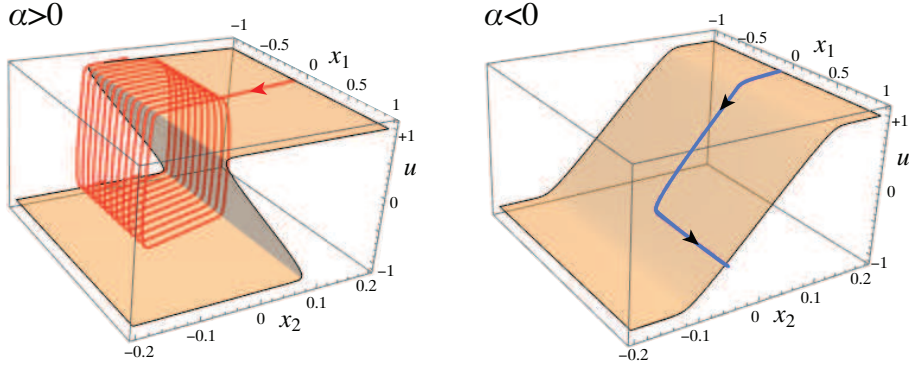}
\vspace{-0.3cm}\caption{\footnotesize\sf A simulation of \eref{slow} where the $(x,y)$ system is as given in \eref{exutkin}, for $\alpha=\pm 0.1$ and $\eps=0.01$, showing a typical trajectory, and the slow manifold $\Phi-u=0$.}\label{fig:ut2}
\end{figure}
%








\subsection{A final curiosity}\label{sec:isochrone}


We end with an interesting note concerning the curve
$$
\op Q=\cc{\;(x,y,u):\;g\bb{x,0;u}=0,\;u=\phi\bb{\mbox{$\frac{y+\al u}{\eps}$}}\;},
$$
where $\eps=\kappa \al$, on which \eref{Ns} becomes a one-dimensional system in $x$, following Utkin's dynamics.
In the proof to Theorem \ref{thm:utkinalpha} (see Appendix \ref{sec:uproof}), we find for $\alpha<0$ that the curve $\op Q$ plays a key role,
by creating an attracting invariant manifold where Utkin's dynamics occurs.

In the case $\alpha>0$ the  curve $\op Q$ is a repeller,
and  therefore it does not play any role in the hysteretic (Filippov) dynamics,
but as the following result shows, it does have topological significance.
\begin{lemma}
The fast isochrone.
Consider system  \eref{Ns}  where  and $u=\pm 1$.
There exists a curve $\cal I$ that is the isochrone of the regularization region boundaries $y=\pm \al$, meaning that the flow of \eref{Ns} with $u=-1$
takes an equal amount of time to reach  $y=+\al$ than the flow of \eref{Ns} with $u=+1$ needs to reach  $y=-\al$ from $\cal I$.
If $g$ is linear in $u$,
then the manifold $\cal I$ and the projection of $\op Q$ in the $(x,y)$ plane coincide up to $\ord{\al^2,\eps}$.
\end{lemma}

\begin{proof}
Take an initial point $p$ with coordinates $(x,y)=(x_p,y_p)$ such that $|y_p|<\al$.
Approaching from $y$ negative (along the $f^-,g^-$ systems) the time taken to reach $y=+\al$ is $\Delta t^-$
such that
$$
\int_0^{\Delta t^-}g^-(x^-(t), y^-(t))dt=\al -y_p
$$
while approaching from $y$ positive (along the $f^+,g^+$ systems) the time taken to reach $y=-\al$ is
$$
\int_0^{\Delta t^+}g^+(x^+(t), y^+(t))dt=-\al -y_p\;.
$$
Applying the mean value theorem we have that, there exist $t^\pm \in (0,\Delta t^\pm)$ sicu that
$$
\int_0^{\Delta t^\pm}g^\pm(x^\pm(t), y^\pm(t))dt= g^\pm(x^\pm(t^\pm), y^\pm(t^\pm))\Delta t^\pm,
$$
therefore both times are equal if
$$
\frac{-\al -y_p} {g^+(x^+(t^+), y^+(t^+))}
=
\frac{\al -y_p} {g^-(x^-(t^-), y^-(t^-))}\;.
$$
This defines the isochrone surface $\op I$. Now, taking the limit when $\al \to 0$ one obtains
$$
y_p=\al \frac{g^+(x_p,0)+g^-(x_p,0)}{g^+(x_p,0)-g^-(x_p,0)}+\OO(\al^2).
$$
Let us now consider that $g$ in \eref{Ns} is linear in $u$, that is
$$
g=\frac{g^++g^-}{2}+\frac{g^+-g^-}{2}u .
$$
Now let us find $\op Q$ to show that it  coincides in first order with $\op I$.
For $u\in[-1,1]$ we have that the curve $u=\phi(\frac{y+\al u}{\eps})$ is contained in $y=-\al u+\OO(\eps)$.
If the vector field is linear with respect to $u$, from $g(x,0,u)=0$ one easily obtains
$u= -\frac{g^+(x,0)+g^-(x,0)}{g^+(x,0)-g^-(x,0)}$ which, combined with the expression for $y$ obtained above gives
$$
y=\alpha\frac{g^+(x,0)+g^-(x,0)}{g^+(x,0)-g^-(x,0)}+\OO( \eps)\;.
$$
Therefore, both curves coincide up to $\ord{\eps,\alpha ^2}$.
\end{proof}

\section{Closing Remarks}\label{sec:conc}

The two canonical formalisms for handling the discontinuity are mainly associated with the names of Filippov \cite{f64,f88} and Utkin \cite{u77,u74,u92}.
Both methods are intuitive, but one expresses the system on the switching manifold in terms of the component vector fields $\bb{f^\pm(x,y),\;g^\pm(x,y)}$,
the other in terms of a combination $\bb{f(x,y;u),\;g(x,y;u)}$.
The latter permits nonlinearity in the switch (i.e. in the $u$ dependence),
and it turns out that either linear or nonlinear models can both be proven `rigorously' to
approximate the dynamics of a system specified by \eref{Ns}.
With increasing applications of interest in the mechanical, biological, or social sciences, clearer criteria for choosing between the two methods are clearly desirable.

The process of regularizing a discontinuity is widely assumed to support Filippov's method,
when actually the process is tautologous: the way one chooses to regularize the vector field actually pre-determines
whether the outcome will be dynamics that assumes a linear combination across the discontinuity, or permits nonlinearity.
Fortunately the situation is much less ambiguous than this would suggest, and as we have shown, Filippov's linear sliding and the (less common)
nonlinear sliding are each valid in certain distinct limits.

The results here apply to a single attracting switching surface. 
The situation for two or more switches turns out to be even richer and more intriguing, see \cite{j16jitter}.




\clearpage

\appendix




\newpage
\section{Proof of Theorem \ref{propfil}: hysteresis gives linear sliding to $\ord{\alpha}$
}\label{sec:hproof}

Take $\alpha_1 >0$ fixed.
We will take a compact set $\KK=[-M,M]\times [-\alpha _1, \alpha _1]$, where $M$ is given in \eqref{g},
and consider the vector field \eqref{Ns}, \eqref{hyst0} and \eqref{FF}, that we denote as
\begin{equation}\label{def:Filippov}
Z(x,y)=\left\{\begin{array}{l}
    X^+(x,y),\, (x,y)\in \KK^+\\
    X^-(x,y),\, (x,y)\in \KK^-,
    \end{array}\right.
\end{equation}
where $X^\pm = (f^\pm,g^\pm)$ as in \eqref{FF} and $\KK^+=\{(x,y)\in \KK,\ y\ge 0\}$,
$\KK^-=\{(x,y)\in K,\ y\le 0\}$ with a switching surface
$$
\Sigma=\{(x,y) \in \KK, \ y= 0\}.
$$
We
know that $g$ satisfies  \eqref{g}
therefore:
$$
g^- (x,y)>0, \quad g^+(x,y)<0.
$$

The first observation is that, after a smooth change of variables given by the flow box theorem, one can assume that
$f^+(x,y)=0$, $g^+(x,y)=-1$, and therefore the upper vector field is
\begin{equation}\label{eq:simplifiedX+}
X^+(x,y)=\left(
\begin{array}{c}
0\\
 -1
 \end{array}
 \right).
\end{equation}
To produce motion along the surface we must then have $f^-\neq0$, and without loss of generality we assume $f^->0$.
Then the Filippov vector field \eqref{xfilippov} in these new variables $(x,y)$ is given by
\begin{equation}\label{eq:simplifiedF}
\dot x= f_F(x):=\frac{f^-(x,0)}{g^-(x,0)+1} >0,
\end{equation}
and therefore the Filippov vector field ``goes to the right''. The case

Assume that, for any $(x,y) \in \KK$, one has the following bounds:
\begin{equation}\label{eq:boundsfg}
\begin{array}{cccccccc}
0&<& D&<& f^-(x,y)&<&C\\
0&<& D&<& g^-(x,y)&<&C.
\end{array}
\end{equation}
During this section, we will use the letter $L$ to denote any constant
just depending on the vector field $X^-$ and its derivatives in the compact $\KK$.

Consider the solution of the Filippov vector field \eqref{eq:simplifiedF} in $\Sigma$, $x_F(t)$
with initial condition $x_F(0)=x_0 \in (-M,M)$ and such that $x_F(t) \in (-M,M)$, for $t\in [0,T]$.

Take $\gamma >0$ small enough and $0<\alpha <\al _1$ such that the rectangles:
$$
D_\gamma ^\alpha =[x_0-\gamma , x_F(T)+\gamma ] \times [-\al,\al],
$$
satisfy $D_\gamma ^\alpha \subset \KK$.

Consider the solution of the vector field \eqref{Ns} using the hysteretic process:
take the solution $(x^-(t),y^-(t))$ of $X^-$ with initial condition $(x^-(0),y^-(0))=(x_0, -\alpha)$
and  $T^-=T^-(x_0)$ such that $y^-(T^-)=\alpha$.
Then  define $\bar x_0= x^-(T^-)$.
It is clear that the function $T^-$ also depends on $\al$ but we avoid this dependence if there is not danger of confusion.
Now, consider the solution
$(x^+(t),y^+(t))$ of $X^+$ with initial condition $(x^+(0),y^+(0))=(\bar x_0, \alpha)$,
and  $T^+=T^+(\bar x_0)$ such that $y^+(T^+)=-\alpha$.

Then  define $ x_1= x^+(T^+)$.
This completes a cycle of the hysteretic process.

It is important to note that for the vector field $X^+$ given in
\eqref{eq:simplifiedX+} one has that
$T^+= 2\al$ and $x_1= \bar x_0=x^-(T^-)$.
Therefore, after one cycle of the hysteretic process, the hysteretic solution
$(x_h(t),y_h(t))$ gets to the point $(x_1,-\al)$, with $x_1= x^-(T^-)$,
and the time spent in the cycle is $S=S(x_0)=T^-(x_0)+2\al$,
that is:
\begin{equation}\label{eq:histereticsol}
x_h(S(x_0))=x_h(T^-(x_0)+2\al)= x^-(T^-(x_0)).
\end{equation}
Proceeding by induction one can define
$x_i =x_h(S(x_{i-1}))= x^- (T^-(x_{i-1}))$,
where the time $S(x_{i-1})= T^-(x_{i-1})+2\al$, and $T^-(x_{i-1})$ is the time needed by the solution
$(x^-(t),y^-(t))$ of $X^-$ with initial condition
$(x^-(0),y^-(0))=(x_{i-1},-\al)$ to arrive at $y=\al$, that is, $y^-(T^-(x_{i-1}))=\al$.

We can use the hysteretic process to move along the rectangle $D_\gamma ^\alpha$.
Next proposition, which gives immediately Theorem \ref{propfil},
relates the resulting trajectory with the one obtained in $\Sigma$ following the Filippov vector field.

\begin{proposition}\label{prop:error}
Fix $T >0$ and consider the solution $x_F(t)$ of the Filippov System \eqref{eq:simplifiedF} in $\Sigma$,  for $0\le t\le T$
and the hysteretic solution $(x_h(t), y_h(t))$ with initial condition
$(x_h(0), y_h(0))=(x_0, -\al)=(x_F(0), -\al)$.

Take $n=n(\al)$ the number of cycles of the hysteretic solution
such that $x_n \le x_F(T)\le x_{n+1}$.

Then there exists a constant $L$ only depending of the vector field $X^-$ and the compact $\KK$ such that
$$
|x_n-x_F(T)| \le L \al.
$$
Moreover, for any $t \in [0,T]$
$$
|x_h(t)-x_F(t)| \le L \al.
$$

\end{proposition}

To prove this proposition, which reminds the  estimation of the error in the Euler method, we first need some lemmas.

\begin{lemma}\label{lem:histeresis1}
Let
$\bar x \in [x_0-\frac{\gamma }{2}, x_F(T)+\frac{\gamma }{2}]$.
Then $\exists \alpha _0 $, with $0<\alpha_0<\alpha _1$, such that if $0<\al \le \al_0$,
the solution $(x^-(t),y^-(t))$ of $X^-$ with initial condition $(x(0), y(0))=(\bar x,-\al)$
reaches $y=\al$ in a point $(x^*,\al) $, with $x^* \in (x_0-\gamma , x_F(T)+\gamma )$.
\end{lemma}

\begin{proof}
As $g^- >0$ and $f^->0$, the lower bound is already fulfilled.
To prove the upper bound, from the equation for the orbits of the vector field $X^-$ we have
$$
y(x)+\al =\int _{\bar x}^{x}\frac{g^-(s,y(s))}{f^-(s,y(s))} ds .
$$
Using the bounds \eqref{eq:boundsfg} one has we have
$$
y(x)+\al >\frac{D}{C}(x-\bar x) .
$$
Using that $\frac{D}{C}(x^*-\bar x) < 2\alpha $
we obtain the desired result taking $0<\alpha \le \alpha _0<\frac{D\;\gamma}{4C}$.
\end{proof}

From now on, we will write $h=\OO(\alpha^n)$ when $h$ is a function bounded as
$|h|\le L\al ^n$, and $L$ is, as usual, a constant only depending on  the vector fields $X^\pm$ and their derivatives in the compact $\KK$.

Next lemma gives a first upper bound of the transition time $T^-$.

\begin{lemma}\label{lem:histeresis2}
Let $\alpha_0$ as  in Lemma \ref{lem:histeresis1}, and $0<\alpha \le \alpha_0$.
Let $\bar x \in [x_0-\frac{\gamma }{2}, x_F(T)+\frac{\gamma }{2}]$.
Let $T^-=T^-(\bar x)$ the time needed for the solution $(x^-(t),y^-(t))$ of $X^-$ with initial condition $(x(0), y(0))=(\bar x,-\al)$
 to reach $y=\al$.
Then there exists $L>0$ such that
\begin{equation}\label{eq:boundT}
0<T^-(\bar x) <L \al
\end{equation}
\end{lemma}

\begin{proof}
We know, by lemma \ref{lem:histeresis1}, that the solution $(x^-(t),y^-(t))$ of $X^-$ remains in $D_\gamma^\al$ until it reaches $y=\al$.
As:
\begin{equation}\label{eq:tfc}
\int _0^{T^-}\dot y^-(t)dt= \int _0^{T^-} g^-(x^-(t),y^-(t))dt
\end{equation}
one has, using the definition of $T^-$, the bounds \eqref{eq:boundsfg} and lemma \ref{lem:histeresis1}:
$$
2\alpha = g^-(x^-(t^*),y^-(t^*)) T^->D T^-
$$
\end{proof}

\begin{lemma}\label{lem:histeresis3}
With the same hypotheses of Lemma \ref{lem:histeresis2} one has:
\begin{equation}\label{eq:asymptoticsT}
T^-(\bar x)=\frac{2 \al}{g^-(\bar x,0)}+ \OO(\al^2)
\end{equation}
\end{lemma}

\begin{proof}
By the mean value theorem one has:
$$
g^-(x^-(t),y^-(t))=g^-(\bar x, -\al) + t G(t^*),
$$
where
$
G(t)= \frac{d}{dt}g^-(x^-(t),y^-(t))
$
and $t^*=t^*(t)$ satisfies $0\le t^* \le t$.

Using bounds \eqref{eq:boundsfg} and Lemma \ref{lem:histeresis1}, one has that $|G(t^*)| \le L$.
Then we have, using again \eqref{eq:tfc}:
\begin{eqnarray*}
2\al &=& T^-
g^-(\bar x,-\al) +
\int _0^{T^-}t G(t^*) dt \\
&=& T^- g^-(\bar x, 0)- T^-\frac{\partial g^-}{\partial y}(\bar x, \al ^*)\al+
\int _0^{T^-} t\,G(t^*)\, dt \\
&=& T^- g^-(\bar x, 0)+ m (T^-,\al)
\end{eqnarray*}
where $0\le \alpha ^*\le \alpha$.
Then, by the a-priori bounds on $T^-$ given in lemma \ref{lem:histeresis2} and the bound of $G(t^*)$,
and using again bounds like \eqref{eq:boundsfg} for the derivatives of $g^-$, one has that there exists $L>0$, such that
$|m(T^-,\al)| \le L \alpha ^2$,
and therefore one gets that
$$
T^-=\frac{2\al}{g^-(\bar x,0)}+\OO(\al ^2)
$$
\end{proof}

\begin{remark}
For $0\le t\le S(x_0)$, where $S(x_0)=T^-(x_0)+2\alpha $ is the time needed in a hysteretic cycle,
the solutions of the Filippov vector field satisfy $x_F(t)-x_F(0)= \OO(\al)$, and the hysteretic solution
also satisfies $x_h(t)-x_h(0)= \OO(\al)$, consequently:
$$
x_F(t)-x_h(t)=\OO(\alpha), \
0\le t\le S(x_0).
$$
Next lemma says that at the end point $S(x_0)$ the solutions approach each other up to order $\al^2$.
Therefore, the new hysteretic cycle begins $\al^2$ close to the Filippov solution at every step.
\end{remark}

\begin{lemma}\label{lem:histeresis4}
Let $\alpha_0$ as given in Lemma \ref{lem:histeresis1}, and $0<\alpha \le \alpha_0$.
Let $\bar x \in [x_0-\frac{\gamma }{2}, x_F(T)+\frac{\gamma}{2}]$.
Consider the solution $\bar x_F(t)$ of the Filippov vector field \eqref{eq:simplifiedF}, with initial condition $\bar x_F(0)=\bar x$.
Let $(x_h(t),y_h(t))$ the hysteretic solution  with initial condition $(x_h(0),y_h(0)=(\bar x,-\al)$.
Let be $S=S(\bar x)=T^-(\bar x)+2\al$ the time in a hysteretic cycle, where $T^-=T^-(\bar x)$ is given in Lemma \ref{lem:histeresis3}.
Then
\begin{equation}\label{eq:asymptoticsT}
\bar x_1-\bar x_F(S)=x_h(S)-\bar x_F(S)=\OO(\al^2).
\end{equation}

\end{lemma}
\begin{proof}
The proof is an easy consequence of lemma \ref{lem:histeresis3} and the Taylor theorem applied to both solutions.
On the one hand the Filippov solution satisfies equation \eqref{eq:simplifiedF}, and therefore:
$$
\bar x_F(t)= \bar x+ \frac{f^-(\bar x,0)}{1+g^-(\bar x,0)} t+ \OO(t^2)
$$
and
\begin{eqnarray*}
\bar x_F(T^-+2\al) &=& \bar x+ \frac{f^-(\bar x,0)}{1+g^-(\bar x,0)} \left(\frac{2\al}{g^-(\bar x,0)}+\OO(\al ^2)+2\al\right)+ \OO(\al^2)\\
&=& \bar x+\frac{f^-(\bar x,0)}{g^-(\bar x,0)}2\al + \OO(\al^2).
\end{eqnarray*}
We have, using the equations of $X^-$:
\begin{eqnarray*}
\bar x_1 &:=& x_h(T^-+2\al) = x^-(T^-)= \bar x+ f^-(\bar x,-\al) T^-+ \OO(\al ^2) \\
&=&\bar x+\frac{f^-(\bar x,0)}{g^-(\bar x,0)}2\al + \OO(\al^2).
\end{eqnarray*}

Therefore we have
$$
|\bar x_F(T^-+2\al)-x_h(T^-+2\al) | =\OO(\al^2),
$$
uniformly in $D^\al _\gamma$.
\end{proof}

Next lemma gives  the number of cycles needed to reach the final position of the Filippov solution
$x_F(T)$.

\begin{lemma}\label{lem:numbercycles}
Consider the Filippov solution $x_F(t)$, $0\le t\le T$, with initial condition $x_F(0)=x_0$.
Consider also the hysteretic solution $(x_h(t), y_h(t))$ with initial condition
$(x_h(0), y_h(0))=(x_0,-\alpha)$.

Let
 $ n=n(\al)$ be the number of hysteretic cycles such that:
$$
x_n\le x_F(T) \le x_{n+1}
$$
where $x_i$ is the value of the $x$ coordinate of the $i$ hysteretic cycle.

Then $n(\al)= \OO(\frac{1}{\al})$, uniformly for $0<\al\le \al _0$.
\end{lemma}

\begin{proof}
Let $x_i$ denote the value of $x$ on the $i$- cycle.
By lemma \ref{lem:histeresis3} we know that the time needed by the orbit of $X^-$
with initial condition $(x^-(0),y^-(0))=(x_i,-\al)$ to get $y=\alpha $ is
$T^-(x_i)= \frac{2\al}{g^-(x_i,0)}+\OO(\al^2)$ and the time of the corresponding orbit of
$X^+$ to come back to $y=-\al$ is $2\al$.
Moreover, we know that $x_{i+1} =x^-(T^-(x_i))$.

Moreover, using bounds \eqref{eq:boundsfg}
$$
D T^-(x_i) \le x_{i+1}-x_i \le CT^-(x_i),
$$
but, by lemma \ref{lem:histeresis3} and bounds \eqref{eq:boundsfg} we know that there exists $L_1$, $L_2$ such that
$\alpha L_2 \le T^-(x_i)\le \al L_1$, and therefore, uniformly in $\al$ we get:
$$
D L_2 \al \le x_{i+1}-x_i\le C L_1 \al
$$
adding these inequalities from $i=0,\dots , n$, $n=n(\al)$, one obtains:
$$
L_2 D \al n(\al) \le x_{n(\al)}-x_0\le L_1 C \al n(\al).
$$

In particular
$$
n(\al) \le \frac{x_{n(\al)}-x_0}{L_2 D \al} \le
\frac{x_F(T)-x_0}{L_2 D \al}
$$
obtaining that $n(\al) =\OO(\frac{1}{\al})$.
To get a lower bound for $n(\al)$ we use the inequality for $n+1$, obtaining:
$$
n(\al) +1 > \frac{x_{n+1}-x_0}{L_1 C\al}\ge \frac{x_F(T)-x_0}{L_1 C\al}.
$$
\end{proof}


Next lemma is devoted to bound the error:

$$
\eps _i= x_F(S(x_0)+S(x_1)+\cdots +S(x_{i-1}))-x_i
$$
where $x_F(t)$ the solution of the Filippov system \eqref{eq:simplifiedF},
$S(x_l)=T^-(x_l)+2\al$ is the time needed in the $l$- hysteretic cycle,
and $T^-(x_l)$ is the time needed by the solution of $(x^-(t), y^-(t))$ with initial condition $(x_l, -\al)$ to get to $y=\al$.

\begin{lemma}\label{lem:induction}
The error at the $i$-cycle, $1\le i \le n(\al)$ satisfies:
$$
|\eps _i| \le (1+\al L) |\eps _{i-1} + L\al ^2
$$
where $L$ is uniform in the compact $\KK$
\end{lemma}

\begin{proof}
We denote as:
$$
\bar T_i = T ^-(x_0)+\dots +T^- (x_{i-1}), \ \mbox{and} \ \bar S_i=\bar T_i +i 2\al, \quad i\ge 1
$$
Note that, for $i\ge 2$:
$\bar S_i -\bar S_{i-1}= S(x_{i-1})=T^- (x_{i-1})+2\al$.

We must estimate
$$
\eps _i= x_F(\bar S_i) - x_h(\bar S_i)
$$

Using Taylor's theorem, one has:
$$
x_F(\bar S_i) =x_F(\bar S_{i-1}) +\dot x_F(\bar S_{i-1})(T^-(x_{i-1})+2\al)+
\frac{\ddot x_F(\bar T_{i}^*)}{2}(T^-(x_{i-1})+2\al)^2 .
$$
with
$\bar S_{i-1} \le \bar T_i^* \le \bar S_{i}$, for any $1\le i \le n(\al)$.
Moreover
$$
\dot x_F(\bar S_{i-1})=
f_F(x_F(\bar S_{i-1}))
= f_F(x_{i-1})+f'_F(\xi) \eps _{i-1}, \ x_{i-1}\le \xi \le x_F(\bar S_{i-1})
$$
where  $f_F$ is given in \eqref{eq:simplifiedF}. Then:
\begin{eqnarray}
x_F(\bar S_i) &=&
x_F(\bar S_{i-1}) +
\frac{f^-(x_{i-1} ,0)}{1+g^-(x_{i-1},0)}
(T^-(x_{i-1}) +2\al)\nonumber \\
&+&
M_F .
\label{xf}
\end{eqnarray}
Where
\begin{equation}\label{MF}
M_F=
f'_F(\xi)\eps _{i-1} (T^- (x_{i-1})+2\al) +\frac{\ddot x_F(\bar T_i^*)}{2}(T^- (x_{i-1})+2\al) ^2
\end{equation}
and therefore, by bounds \eqref{eq:boundsfg}
$$
|M_F| \le L(\eps _{i-1}\al + \al^2).
$$

Now we proceed analogously with the hysteretic solution.

We know that $x_i:=x_h(\bar T _i + i 2\al ) = x^-(T^- (x_{i-1}))$, where $(x^-(t), y^-(t))$ is the solution of $X^-$ with initial condition $(x_{i-1}, -\al)$.
We can use again a Taylor's theorem:
\begin{eqnarray}
x_i&=& x^-(0)+\dot x^-(0) T^- (x_{i-1})+\frac{\ddot x^-(\eta)}{2}(T^- (x_{i-1}))^2 \nonumber
\\
&=&x_{i-1}+f^-(x_{i-1},-\al)T^- (x_{i-1})+\frac{\ddot x^-(\eta)}{2}(T^- (x_{i-1}))^2 \nonumber \\
&=&
x_{i-1}+f^-(x_{i-1},0)T^- (x_{i-1}) + M_h \label{xh}
\end{eqnarray}
where $0\le \eta\le T^- (x_{i-1})$, and therefore
\begin{equation}\label{Mh}
M_h =-\frac{\partial f^-}{\partial y} (x_{i-1},\al ^*)\al T^- (x_{i-1})+
\frac{\ddot x^-(\eta)}{2}(T^- (x_{i-1}))^2
\end{equation}
with $0\le \al ^*\le \al$,
and
$$
|M_h|\le L \al^2.
$$

Subtracting \eqref{xf} and \eqref{xh}
we obtain:
$$
\eps _i =\eps _{i-1} + \frac{f^-(x_{i-1},0)}{1+g^-(x_{i-1},0)}(T^- (x_{i-1})+2\al)-f^-(x_{i-1},0)T^-(x_{i-1})
+M_F-M_h
$$
and, using the formula for $T^-(x_{i})$ given in \eqref{eq:asymptoticsT}, we have
$$
|\frac{f^-(x_{i-1},0)}{1+g^-(x_{i-1},0)}(T^- (x_{i-1})+2\al)-f^-(x_{i-1},0)T^-(x_{i-1})| \le L\al ^2
$$
which gives:
$$
|\eps _i|\le (1+L\al) |\eps_{i-1}| + L\al ^2.
$$
\end{proof}

\noindent{\textbf{Proof of proposition \ref{prop:error}}}

The result of the Lemma \ref{lem:induction} gives:
$$
|\eps_i|\le \left( (1+L\al)^i-1\right)\al \le \left( (1+L\al)^{n(\al)}-1\right)\al
$$
Now, using the estimate for $n(\al)$ given in lemma \ref{lem:histeresis4} we get
$$
\eps _{n(\al)} \le
\left( (1+L\al)^{\frac{x_F(T)-x_0}{L_2D\al}}-1\right)\al
$$
and using that $\lim _{\al \to 0}(1+L\al)^{\frac{\beta}{\al}}=e^{\beta L}$
one gets that there exists $L>0$ such that
$$
\eps _{n(\al)} \le
L\al .
$$
Let's $t\in [0,T]$, take $0\le i\le n(\al)$ such that $t\in [\bar S_i,\bar S_{i+1}]$. We have then that $t=\bar S_i+ \OO(\al)$, and therefore:
$$
x_F(t)= x_F(\bar S_i) +\OO(\al), \quad x_h(t)=x_i+ \OO(\al).
$$
Using that $\eps_i= x_F(\bar S_i) -x_i =\OO(\al)$, extends the estimates for all $t$.
\qed

\newpage
\section{Proof of Theorem \ref{proput}: smoothing the step in equation \eref{hyst0} gives nonlinear sliding to $\OO(\al)$ }
\label{sec:sproof}

Let $x_0=x_U(0)\in(-M,+M)$ and $y_0\in\sq{-\alpha,\alpha}$,
and consider the smooth system \eref{Ns} where
$u=\phi(y/\alpha)$ using the function $\phi$ defined in \eref{phi}.
We have
\begin{equation}
\begin{array}{rcl}
\dot x&=&f(x,y;\phi(y/\alpha))\;\\
\dot y&=&g(x,y;\phi(y/\alpha))\;
\end{array}
\end{equation}
Let $v=y/\alpha$ to obtain the slow subsystem
\begin{equation}
\begin{array}{rcl}
\dot x &=& f(x,\alpha v;\phi(v))\;\\
\alpha\dot v &=&g(x,\alpha v;\phi(v))\;
\end{array}
\end{equation}
with critical limit
\begin{equation}\label{eq:criticalimit}
\begin{array}{rcl}
\dot x&=&f(x,0;\phi( v))\\\
0&=&g(x,0;\phi( v))\ \;
\end{array}
\end{equation}
Consider also the fast subsystem, obtained by denoting the derivative with respect to $\tau= t/\alpha$ with a prime, so
\begin{equation}\label{eq:fastsubsystem}
\begin{array}{rcl}
x'&=&\alpha f(x,\alpha v;\phi( v))\;\\
 v'&=&g(x,\alpha v;\phi( v))\;
\end{array}
\end{equation}
Then assuming $\phi'( v)>0$, and, by \eqref{g}, $\partial g(x,0;u)/\partial u<0$, which imply
\begin{equation}
\frac{\partial\;}{\partial v}g(x,0;\phi( v))=\phi'( v)\frac{\partial\;}{\partial u}g(x,0;\phi( v))<0
\end{equation}
by the inverse function theorem there exists a graph $ v=\gamma_0(x)$ such that
$$
0=g(x,0;\phi(\gamma_0(x)))\;,
$$
and a critical manifold
\begin{equation}
\op U^0=\cc{(x, v):\; v=\gamma_0(x)\;, |x|\le M}\;.
\end{equation}
Moreover, the dynamics of system \eqref{eq:criticalimit} in this manifold is exactly the Utkin equivalent control of Definition \ref{def:utkin}: $x_U(t)$.

$\op U^0$ is the set of equilibria of the fast subsystem \eqref{eq:fastsubsystem} in the critical limit $\alpha=0$, satisfying the system
\begin{equation}
\begin{array}{rcl}
x'&=&0\\
v'&=&g(x,0;\phi(v))\;
\end{array}
\end{equation}
and it is an
attracting normally hyperbolic manifold of the one-dimensional system in $v$,
since $\frac{\partial\;}{\partial v}g(x,0;\phi(v))<0$.
Hence by Fenichel Theorem for $\alpha>0$ there exist invariant smooth manifolds $\op U^\alpha$ which lie $\alpha$-close to $\op U^0$.
More concretely:
\begin{equation}
\op U^\al=\cc{(x, v):\; v=\gamma(x;\al)\;, |x|\le M}\;, \ \gamma (x,\al)=\gamma_0(x) +\OO(\al).
\end{equation}
Take a solution $(x(t), v(t))$, where $(x(0), v(0))=(x_0, v_0)$, $v_0 \in [-1,1]$.
As the slow vector field points inwards on the borders $v=\pm1$, one can easily see that the solutions enter the basin of exponential attraction by the Fenichel
manifold (see \cite{BonetS16}), and therefore one has that there exists constants $L_1>0$, $L_2>0$,
$$
|x(t)-x_\al(t)| \le L_1 e^{-L_2 t/\al}, \ t\ge 0
$$
where $(x_\al (t), \gamma (x_\al (t),\al))$ is the solution along the Fenichel manifold begining at $(x_0,  \gamma (x_0,\al))$.
Now, using that $x_\al(t)=x_U(t)+\OO(\al)$ and going back to variables $(x,y)$, with $y= \al v$, we obtain the desired result;
a solution of the smoothed system $\bb{x(t),y(t)}$ for $t\in\sq{0,T}$ such that $x(0)=x_0$ and $y(0)=y_0$ and $|y_0|\le\al$ satisfies
\begin{equation}
x(t)=x_U(t)+\OO(\alpha)\;.
\end{equation}

\newpage
\section{Proof of Theorem \ref{thm:filalpha}: relaxation gives linear sliding to $\ord{\alpha}$ for $\eps,\alpha>0$.}\label{sec:fproof}


Take $\alpha_1 >0$ and $0<\delta_1<\frac{1}{4}$ fixed.

We will take the variables $(x,y,u)$, where $y=\alpha v$, in a compact set
$$
\KK=[-M,M]\times [-\alpha_1,\alpha_1] \times [-1-\delta_1,1+\delta_1].
$$
During the proof we will consider solutions of system
\eref{slow} with $y= \alpha v$ which never leave this compact set.
Therefore we can assume that there exists a constant $C$ such that, for $\al <\al_1$:

\begin{equation}\label{fitesfgK}
|f(x,y,u)|\le C, \ |g(x,y,u)|\le C, \ \forall (x,y,u) \in \KK
\end{equation}

Moreover, during this proof we will denote by $L$ any constant only depending of the vector field \eqref{bslow} and its derivatives in this compact.

We will assume, by the hypotheses \eqref{g} on $g$, that this function changes its sign at $u=0$, which is not a restrictive assumption.
Therefore, one can ensure that
\begin{eqnarray}
&& g(x,y,u)\ge 0, \ \forall (x,y,u) \in \KK, \ -1-\delta_1 \le u \le 0 \label{fitesgsigne}
\\
&& g(x,y,u)\le 0, \ \forall (x,y,u) \in \KK, \
0\le u \le 1+\delta_1 \nonumber
\end{eqnarray}

\subsubsection{A positively invariant annulus}

We will define a subset of $\KK$ such that the vector field \eref{slow} points inwards in all its borders except, eventually, at $x=\pm M$.
This will allow us to control the solutions.

During this section we will take $0<\al<\al_1$ small enough, $\eps = \kappa \al$ for some $0<\kappa<\frac{1}{4}$ small
and $0< \dd\le \delta_1$, in such a way that
one has uniform bounds $\forall (x,y,u) \in \KK$:
\begin{eqnarray}
0&<& D\le g(x,y,u) \le C, \
-1-\delta_1\le u \le
-1+\dd+2\kappa \label{fitesfgsigne+}\\
-C&\le & g(x,y,u)< -D <0, \
1-\dd-2\kappa\le u \le 1+\delta_1 \label{fitesfgsigne-}
\end{eqnarray}

To proof Theorem \ref{thm:filalpha} (and later Theorem \ref{thm:utkinalpha}),
we introduce a scaled variable $v=y/|\alpha|$ and using \eref{kappa} to eliminate $\eps$ in \eref{slow},
we will work with the following system:
\begin{eqnarray}\label{bslow}
\begin{array}{rll}
\dot x&=&f\bb{x,|\alpha| v;u}\;,\\
|\alpha|\dot v&=&g\bb{x,|\alpha| v;u}\;,\\
\kappa\alpha\dot u&=&\phi\bb{\frac{u+{\rm sign}(\alpha)v}\kappa}-u\;.
\end{array}
\end{eqnarray}


Consider the planes $v+ u=\pm \kappa$.
These planes play a crucial role in the dynamics because bellow the plane $v+ u=-\kappa$ the function $\phi(\frac{u+v}{\kappa})=-1$ and
therefore the equation for the variable $u$ is given by
$$
\kappa \al \dot u =-1-u.
$$
Analogously, above $v+ u= \kappa$
is given by
$$
\kappa \al  \dot u =1-u .
$$

The situation is then the following:
The sets
$$
\{(x,v,u), -1 +\kappa < v, \  |x|< M, \ u=1\}
$$
and
$$
 \{(x,v,u), v<1 +\kappa, \ |x|< M, \ u=-1\}
$$
are locally invariant by the flow of system
\eref{bslow} and attracting.

To define a positively invariant annulus, the first observation is that, as $-1\le \phi\le 1$,  the planes $u=\pm (1+\dd)$
confine the flow in $-1-\dd \le u\le 1+\dd$.

We will now build an invariant annulus.
Choose a value of $K>0$ big, depending only of the bounds \eqref{fitesfgK} but independent of $\kappa$, $\alpha$ and $\dd$.
We will fix $K$ satisfying these conditions in next Proposition.
Consider the following plane segments:
\begin{equation}
\begin{array}{rcl}
r_1&\!\!\!\!=\!\!\!\!&\{(x,\alpha v,u) \in \KK \ -1-\dd-\kappa \le v \le 1+ \dd +\kappa +K\kappa , \ u=1+\dd\} \\
r_2&\!\!\!\!=\!\!\!\!&\{(x,\alpha v,u) \in \KK \ u-1-\dd = \frac{2(v+1 +\dd+\kappa)}{K\kappa},\, -1+\dd \le u \le 1+\dd\} \\
r_3&\!\!\!\!=\!\!\!\!&\{(x,\alpha v,u) \in \KK \ v=-1 - \dd -\kappa-\kappa K , \ -1-\dd \le u\le -1+\dd\}\\
r_4&\!\!\!\!=\!\!\!\!&\{(x,\alpha v,u) \in \KK \ -1 - \dd -\kappa-\kappa K \le v \le 1+ \dd + \kappa , \ u=-1-\dd\} \\
r_5&\!\!\!\!=\!\!\!\!&\{(x,\alpha v,u) \in \KK \ u+1+\dd = \frac{2(v-1 - \dd-\kappa)}{K\kappa}, \, -1-\dd \le u\le 1-\dd \}\\
r_6&\!\!\!\!=\!\!\!\!&\{(x,\alpha v,u) \in \KK \ v = 1+\dd +\kappa + \kappa K , \ 1-\dd \le u\le 1+\dd\}
\end{array}
\end{equation}
\begin{proposition}\label{blockinv}
Take $0<\dd\le \delta_1$, $0<\kappa<\frac{1}{4}$ and $K=2C+1$, where $C$ is given in \eqref{fitesfgK}.
Consider the annulus $\bf{A}\subset \KK$ whose exterior border is given by $r_1\cup\dots \cup r_6\cup \{x=\pm M\}$.
Then there exists
$0<\al _0<\al_1$, only depending on the constant $C$ appearing in \eqref{fitesfgK} and $\al_1$,
that for $0<\al \le \al_0$,  and  $\eps =\kappa \al$
any solution of system \eref{bslow} beginning in $\bf{A}$ only can leave it through the borders $x=\pm M$.
\end{proposition}

As any point in the annulus has $v$ coordinates satisfying
$|v|\le 1+\dd+\kappa+\kappa K\le 2+ K$ we choose $\al_0=\frac{\al_1}{2+K}$
and one can ensure that
$g(x,\al v,u)$ verifies bounds \eqref{fitesfgK}
for $(x,v,u)\in \bf{A}$.
The annulus is shown in \fref{fig:block}.
\begin{figure}[h!]\centering\includegraphics[width=1\textwidth]{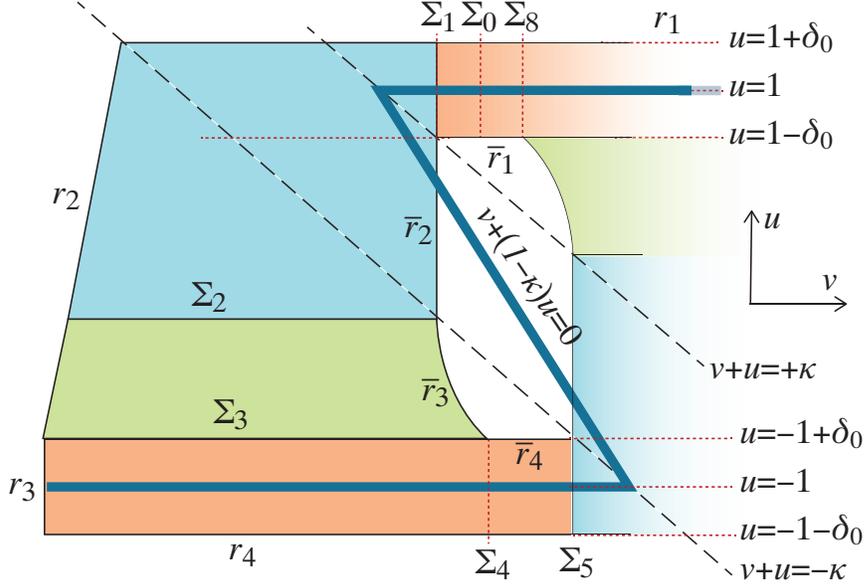}
\vspace{-0.3cm}\caption{\footnotesize\sf The positively invariant annulus $\bf A$.}
\label{fig:block}
\end{figure}

It is clear that the flow points inwards in $r_1$ and $r_4$.
To see that it also points inwards along $r_2$ we need that the scalar product:
\begin{eqnarray*}
A&=&<(0,-\frac{2}{K\kappa},1), (f(x,y,u),\frac{g(x,\al v,u)}{\al}, -\frac{1+u}{\kappa \al })> =\\&&
-\frac{2}{K\kappa \al }g(x,\al v,u)-\frac{1+u}{\kappa \al } <0,
\end{eqnarray*}
for $-1+\dd \le u \le 1+\dd$.
When $0 \le u \le 1+\dd$, we have that:
$$
-2-\dd \le -1-u\le -1
$$
and therefore $
A\le
\frac{2C}{K\kappa \al }-\frac{1}{\kappa \al }
$.
If we take now any $K>2C$ we have that $A<0$. We will choose from now on
\begin{equation}\label{K}
K=2C+1.
\end{equation}
When $-1+\dd \le u\le 0$, we have know that, by \eqref{fitesgsigne} $g\ge 0$ and therefore:
$$
A=
-\frac{2}{K\kappa \al }g(x,\al v,u)-\frac{1+u}{\kappa \al } \le -\frac{1+u}{\kappa \al }\le -\frac{\dd}{\kappa \al } <0.
$$
therefore, $A<0$ along $r_2$.
An analogous reasoning gives that the flow points inwards $\bf {A}$ along $r_5$.

Along $r_3$ we use that $|u+1|\le \dd$, $|x| \le M$ and
$|\al v|
\le \alpha _1$.
Consequently $g(x,\al v,u)$ satisfies \eqref{fitesgsigne} and $\dot y = g(x,\al v,u)>0$ and the flow points inwards $r_3$.
An analogous reasoning gives that the flow points inwards $\bf {A}$ along $r_6$.

\qed

Next step is to build an interior border in $\bf {A}$ to define a positively invariant annulus.

We begin our construction by defining the segment:
\begin{equation}
\bar r_2 = \{(x,\alpha v,u)\in \KK, \ v= -1+ \dd+\kappa , \ 1-\dd-2\kappa\le u \le 1-\dd\}
\end{equation}
Then we take the lower points in $\bar r_2$: $(x,-1+ \dd+\kappa, 1-\dd-2\kappa)$
and we define the next interior border by taking the flow $\varphi(t;x,v,u)$ through these points until it arrives to $u=-1+\dd$.
Observe that, in this region $\eps \dot u= -1-u$ and therefore, one can explicitly compute this time, which independent of the initial value $x$, obtaining
\begin{equation}\label{eq:T}
T_f= -\kappa \al  \log {\frac{\dd}{2-\dd-2\kappa}}.
\end{equation}
Then the equation for the interior border $\bar r_3$ reads:
\begin{eqnarray}
\bar r_3 &=&
\big\{\varphi(t;x,-1+\dd+\kappa, 1-\dd-2\kappa), \ -1+\dd \le u \le 1-\dd-2\kappa, \nonumber\\
&&|x|<M,\ 0\le t\le T_f\big\} \cap \KK
\end{eqnarray}
Now, lets call
$
v_3(x)\le \varphi _v (T_f;x,-1+\dd+\kappa, 1-\dd-2\kappa)
$
the $v$ coordinate of the end points of the surface $\bar r_3$.
Observe that, as the function $g$ satisfies bounds \eqref{fitesfgK}, we have:
$$
v_3(x)\le  -1+\dd+\kappa + \frac{C}{\al}T_f\le -1+\dd+\kappa-C \kappa \log {\frac{\dd}{2-\dd-2\kappa}} :=V_3
$$
Now we can and we define our next border by:
$$
\bar r_4 = \{(x,\alpha v,u)\in \KK, \ \ v_3(x)  \le v \le 1 -\dd- \kappa , \ u=-1+\dd\}.
$$
Analogously to $\bar r_2$ we define the next border:
\begin{equation}
\bar r_5 = \{(x,\alpha v,u)\in \KK, \ v= 1-\dd-\kappa , \ -1+\dd \le u \le -1+\dd+2\kappa \} \\
\end{equation}
Finally, we use again the flow $\varphi(t;x,v,u)$ begining at points $(x, 1- \dd-\kappa, -1+\dd+2\kappa)$ until it arrives to $u=1-\dd$
at time $T_f$ as in \eqref{eq:T} to define the last border:
\begin{eqnarray}
\bar r_6 &=&
\big\{\varphi(t;x,1-\dd-\kappa, -1+\dd+2\kappa), \ -1+\dd+2\kappa \le u \le 1-\dd, \nonumber \\
&&|x|<M,\ 0\le t\le T_f\big\} \cap \KK
\end{eqnarray}
and, calling $v_6(x)$ the $v$-coordinate of the last points on $\bar r_6$ we define the last border:
$$
\bar r_1 = \{(x,\alpha v,u)\in \KK , \ -1+ \dd +\kappa \le v \le v_6 (x), \ u=1-\dd\}.
$$
Analogously to what we did for $v_3(x)$, we can obtain lower bounds for the values of $v_6(x)$:
$$
v_6 (x)\ge 1-\dd-\kappa +C\kappa \log {\frac{\dd}{2-\dd-2\kappa}} :=V_6
$$


Next proposition shows that these sets are interior borders to the flow.
\begin{proposition}\label{annulusinv}
Consider the same hypotheses as in Proposition \ref{blockinv} and that:
$-C\kappa \log \frac{\dd}{2} <1$,

Consider the annulus $\bf{A}$ whose exterior border is given by $r_1\cup\dots \cup r_6\cup \{x=\pm M\}$ and whose interior border is given by.
$\bar r_1\cup\dots \cup \bar r_6\cup \{x=\pm M\}$.

Then for $0<\al \le \al_0$,
any solution of system \eref{bslow} beginning in $\bf{A}$ only can leave it through the borders $x=\pm M$.
\end{proposition}

Let's point out that to $\bar r_1$ to be well defined one needs that
$-1+\dd +\kappa \le v_6(x)$. As $v_6 (x) \ge V_6$, sufficient condition will be:
$$
-1+\dd+\kappa<1-\dd-\kappa+C\kappa \log {\frac{\dd}{2-\dd-2\kappa}}
$$
or, equivalently:
$
-C\kappa \log {\frac{\dd}{2-\dd-2\kappa}} < 2(1-\dd-\kappa)
$.

Using the hypotheses on $\dd$ and $\kappa$
this condition is guaranteed because:
$$
-C\kappa \log {\frac{\dd}{2-\dd-2\kappa}} < -C\kappa \log {\frac{\dd}{2}} <1<2(1-\dd-\kappa).
$$
Moreover, as $\bar r_1$ is above $v+ u =\kappa$ we know that $\dot u >0$ and therefore the flow points inwards $\bf{A}$ along it.
The same reasoning works for $\bar r_4$.

Along $\bar r_2$, $1-\dd-2\kappa\le u \le 1-\dd$ and $|\al v|
\le \al_1$ if we choose $0<\al<\al_0$.
Therefore, by \eqref{fitesfgsigne-} we know that $\al \dot v=g(x,\alpha v,u)<0$, and
we can ensure that
the flow points inwards $\bf{A}$ along $\bar r_2$.
The same reasoning works for $\bar r_5$.
By definition, $\bar r_3$ and $\bar r_6$ are invariant by the flow.

In conclusion, any orbit of system \eref{bslow} which enters in the annulus $\bf{A}$ only can leave it through the borders $x=\pm M$.
\qed

\subsubsection{The Poincar\'{e} map}

Now that we have a positively invariant annulus $\bf{A}$, we will see that the $x$ component of the orbits of system \eqref{bslow} follows closely
the orbits of the Filippov vector field \eqref{fil}.
We will follow closely the proof of Theorem \ref{propfil} where we saw that
hysteretic orbits also follow Filippov ones.
We will define a Poincar\'{e} map inside the annulus whose iterates will correspond to the hysteretic cycles.

Next lemma, whose proof is  straightforward, shows   that solutions beginning in $\bf{A}$ need a finite time, independent of $\alpha$ (and consequently on $\eps$),
to leave $\bf{A}$ though $x=\pm M$.

\begin{lemma}\label{lem:cotainftemps}
Take $\gamma>0$ and any point $(x_0,v_0,u_0)\in \bf{A}$ with $x_0 \in [-M+\gamma, M-\gamma]$,
and call $T_s$ the time needed for the flow of system \eqref{bslow} begining at $(x_0,v_0,u_0)$, to get to $x=\pm M$.

Then one has that:
$$
T_s\ge \frac{\gamma}{C}
$$
where $C$ is the constant given in \eqref{fitesfgK}.
\end{lemma}


Lets now define the following sections which divide $\bf{A}$ in $8$ pieces, see Figure \ref{fig:block}:
\begin{eqnarray*}
\Sigma_0&=&\{ (x,v,u)\in {\bf{A}}, \ |u-1|\le \dd, \ v=0 \}\\
\Sigma_1&=&\{ (x,v,u)\in {\bf{A}}, \ |u-1|\le \dd, \ v=-1 + \dd+\kappa\}\\
\Sigma_2&=&\{ (x,v,u)\in {\bf{A}}, \ u= 1-\dd-2\kappa, \ v<0\}\\
\Sigma_3&=&\{ (x,v,u)\in {\bf{A}}, \ u= -1+\dd, \ v<0\}\\
\Sigma_4&=&\{ (x,v,u)\in {\bf{A}}, \ |u+1|\le \dd, \ v=-1 +\dd +\kappa -C\kappa \log {\frac{\dd}{2-\dd-2\kappa}}\}\\
\Sigma_5&=&\{ (x,v,u)\in {\bf{A}}, \ |u+1|\le \dd, \ v=1 - \dd-\kappa\}\\
\Sigma_6&=&\{ (x,v,u)\in {\bf{A}}, \ u= -1+\dd+2\kappa, \ v>0\}\\
\Sigma_7&=&\{ (x,v,u)\in {\bf{A}}, \ u= 1-\dd, \ v>0\}\\
\Sigma_8&=&\{ (x,v,u)\in {\bf{A}}, \ |u-1|\le \dd, \ v=1 -\dd -\kappa +C\kappa \log {\frac{\dd}{2-\dd-2\kappa}}\}\\
\end{eqnarray*}
and we will consider a Poincar\'{e} map from $\Sigma_0$ to itself. We want to compare an iterate of this map with the solution of the
Filippov vector field at the same amount of time. This is done in next proposition.

\begin{proposition}\label{prop:filalpha}
Consider the solution $x_F(t)$ of \eqref{xfilippov} such that $x_F(0)=x_0$, which satisfies $|x_F(t)|< M$, for $0\le t\le T$.
Take $\alpha_0$ the constant given in Proposition \ref{blockinv}, $0<\kappa <\frac{1}{4}$ and $0<\dd \le \delta_1 $,
such that $C\kappa |\log \frac{\dd}{2}|\le \frac{1}{2}$.
Then there exists $0<\sigma^*\le \al _0$, such that for
$|\al| \le \sigma ^*$, $0<\dd \le \sigma ^*$, $0<\frac{\dd}{\kappa} \le \sigma ^*$,
for any point $z_0=(x_0,0,u_0) \in \Sigma_0$, there exists a time $T^1$ such that $\varphi(T^1;z_0) \in \Sigma _0$.
Moreover:
\begin{itemize}
 \item
$
T^1 = 2\al +\frac{2\al}{ g(x_0,0,-1)} + \OO(\kappa\al , \frac{ \dd} {\kappa}\al, \kappa \al \log{\frac{\dd}{2}},\al ^2).
$
\item
$
x(T^1)=x_0 + 2\al \frac{f(x_0,0,-1)}{g(x_0,0,-1)}+\OO(\kappa\al , \frac{ \dd} {\kappa}\al, \kappa \al \log{\frac{\dd}{2}},\al ^2)
$
\item
$
x_F(T^1)- x(T^1)= \OO(\kappa\al , \frac{ \dd} {\kappa}\al, \kappa \al \log{\frac{\dd}{2}},\al ^2).
$
\end{itemize}
\end{proposition}
After this proposition, and using the same reasoning as in Lemma \ref{lem:numbercycles},
one can see that the number of iterates of the Poincar\'{e} map needed to arrive to $x_F(T)$ is of order
$\OO(\frac{1}{\al})$ and then one obtains the results in Theorem \ref{thm:filalpha}.

We devote the rest of the section to prove the proposition.
Let's call $R_{ij}$ the region in $\bf{A}$ between sections $\Sigma_i$ and $\Sigma_j$.

An important observation is that along $R_{01}\cup R_{12}\cup R_{78}\cup R_{80}$ the function $g$ satisfies \eqref{fitesfgsigne-}.
Analogously, in $R_{34}\cup R_{45}\cup R_{56}$, $g$ satisfies \eqref{fitesfgsigne+}.
Before we proceed with quantitative estimates of this map and of the time needed for the orbit of a point in $\Sigma _0$ to return to it,
we apply the same simplifications to the vector field $X^+$ that we made in the proof of Proposition \ref{prop:error} (Appendix \ref{sec:hproof}).

We can always assume that, after a regular change of variables the vector field
$X^+(x,\al v)=(f^+(x,\al v), g^+(x,\al v)) = (f(x,\al v,1),g(x,\al v,1))$ can be written as
\begin{equation}\label{eq:simplified+}
f(x,\al v,1)=0, \ g(x,\al v,1)=-1
\end{equation}
Therefore, the Filipov vector field will be given by \eqref{eq:simplifiedF}.


In the sequel we will denote by $T^{ij}$ the time needed for a solution to go from $\Sigma _i$ to $\Sigma_j$.
Therefore $T^{ij}=T^{ik}+T^{kj}$, for any $i\le k\le j$.
Next Lemma gives a first estimation of the time spent in a step of the Poincar\'{e} map.

\begin{lemma}\label{lemma:aprioribounds}
Take $\alpha_0$ the constant given in Proposition \ref{blockinv}, $0<\kappa <\frac{1}{4}$ and $0<\dd \le \delta_1 <\frac{1}{4}$,
such that $C\kappa |\log \frac{\dd}{2}|\le \frac{1}{2}$.
There exist constants $L_1$, $L_2$, such that for $0<\alpha\le \al_0$ and for any
$z_0=(x_0,0,u_0)\in \Sigma_0\cap \bf{A}$, if we call $T^1=T^1(z_0)$
the first time such that $\varphi(T^1;z_0)\in \Sigma_0$, one has:
$$
L_1 \al \le T^1\le L_2 \al.
$$
\end{lemma}

For the time $T^+$ spent in the regions $R_{01}\cup R_{12}\cup R_{78}\cup R_{80}$, we use that the maximum variation of
$v$ is between $-(1+\dd+\kappa+K\kappa)$ and $(1+\dd+\kappa+K\kappa)$. Therefore, using that
$
v(T^+)-v_0=\frac{1}{\al}\int _{0}^{T^+}g(x(t), \al v(t),u(t))dt
$
and  the bounds  \eqref{fitesfgsigne-} for $g$ in these regions, one obtain
\begin{eqnarray*}
T^+ &\le& \frac{|v(T^+)-v_0|}{D}\al\le \frac{2(1+\dd+\kappa+K\kappa )}{D}\al\\
&\le& \frac{2(2+K)}{D}\al=\frac{2(3+2C)}{D}\al.
\end{eqnarray*}
The time $T^+$ is bigger that the time $\bar T^+$ spend to cross the region $R_{80}\cup R_{01}$, which is given by
$$
2 -2\dd-2\kappa +C\kappa \log {\frac{\dd}{2-2\kappa-\dd}} = -\frac{1}{\al}\int _{0}^{\bar T^+}g(x(t), y(t),u(t))dt <\frac{C}{\al} \bar T^+,
$$
where we have used that $g$ satisfies \eqref{fitesfgsigne-}. Therefore, using the conditions for $\kappa$ and $\dd$, one has
\begin{eqnarray*}
T^+&\ge& \bar T^+ \ge
 \frac{2 -2\dd-2\kappa+C\kappa \log {\frac{\dd}{2-\dd-2\kappa}} }{C}\al \\
 &\ge& \frac{2-2(\kappa+\dd)-\frac{1}{2}}{C}\al \ge \frac{1}{2C} \al.
\end{eqnarray*}
Similar results give analogous bounds for the time $T^-$ spent in the regions $R_{34}\cup R_{45}\cup R_{56}$.

Finally, the times $T^{23}$ and $T^{67}$ spent to cross
$R_{23}$ and $R_{67}$ are given by \eqref{eq:T} using that in $R_{23}$, $\phi=-1$, and in $R_{67}$, $\phi=1$:
\begin{equation}\label{eq:t23}
T^{23}=T^{67}=-\kappa \al\log \frac{\dd}{2-\dd-2\kappa},
\end{equation}
which gives the upper bound
$$
0<T^{23} =T^{67} \le \eps |\log{\frac{\dd}{2}}|=  \kappa |\log{\frac{\dd}{2}}| \al\le \frac{1}{2C} \al.
$$
Now, we can obtain the upper bounds for $T^1$:
$$
T^1 \le T^++T^-+T^{23} +T^{67}\le (\frac{4(3+2C)}{D}+\frac{1}{C})\al= L_1 \al
$$
And also a lower bound:
$$
T^1\ge \bar T^+ +\bar T^-+T^{23} +T^{67}\ge  (\frac{1}{C}+ 2\kappa |\log{\frac{\dd}{2}}|) \al \ge  \frac{1}{C} \al = L_2 \alpha
$$
which provide the desired bounds.
\qed

\begin{lemma}\label{lem:shorttimes}
With the same hypotheses of Lemma \ref{lemma:aprioribounds},
there exists a constant $L_3$, such that all the times $T^{12}$, $T^{34}$, $T^{56}$, $T^{78}$ satisfy:
$$
|T^{ij}| \le L_3 (\kappa+\dd) \al, \ i,j=1,2 \ \mbox{or } \ i,j=5,6
$$
and
$$
|T^{ij}| \le L_3 (\kappa+\dd-C\kappa \log{\dd/2}) \al, \ i,j=3,4 \ \mbox{or } \ i,j=7,8
$$

\end{lemma}

Let's consider the time $T^{12}$ spent from $\Sigma^1$ to $\Sigma^2$.
In this region $g$ is negative and satisfies bounds \eqref{fitesfgsigne-}.
Moreover, the maximal variation of $v$ smaller than $- (1-\dd)+\kappa- (-(1+\dd)-\kappa -K\kappa)= (K+2)\kappa +2 \dd$, therefore, as $K=2C+1$:
$$
|T^{12}| \le \frac{(K+2)\kappa +2\dd}{C}\al= \frac{(2C+3)\kappa +2\dd}{C}\al \le L_3 (\kappa+\dd) \al
$$
and similar bounds apply to $T^{56}$.

For the time $T^{34}$ we use that the maximal variation of $v$ is smaller than $2 \dd +2\kappa+\kappa K-C\kappa \log{\dd/2}$.
Then using that $K= 2C +1$, and that in this region $g$ is positive and satisfies bounds \eqref{fitesfgsigne+} we obtain
$$
|T^{34}| \le \frac{2 \dd +2\kappa+\kappa K -C\kappa \log{\dd/2}}{D}\al \le L_3 \left( \kappa+\dd - C\kappa \log{\dd/2}\right) \al
$$
and similar bounds apply to $T^{78}$.

Next step is to compute the asymptotics of $T^1$.
From now on, to avoid a cumbersome notation we will use the symbol $\OO(\sigma)$, for $\sigma = \al, \dd,\kappa$ etc
to refer to a function which is bounded by a constant times $\sigma$ for $\sigma \to 0$.


\begin{lemma}\label{lem:first}
With the same hypotheses of Lemma \ref{lemma:aprioribounds},
take $z_0=(x_0,0,u_0) \in \Sigma_0$ and the flow $\varphi(t;z_0)$.
Then there exists $0<\sigma^*\le \al_0$ and a constant $L_4>0$, such that for
$|\al| \le \sigma ^*$, $0<\dd \le \sigma ^*$, $0<\frac{\dd}{\kappa} \le \sigma ^*$, the time $T^{01}$ such that $\varphi(T^{01};z_0)\in \Sigma ^{1}$ satisfies:
$$
|T^{01}- \al| \le L_4 (\kappa+\frac{\dd}{\kappa}+\al)\al
$$
Moreover
\begin{eqnarray}
x(T^{01})&=&x_0 +\OO(\frac{\dd}{\kappa}\al,\al ^2)\\
v(T^{01})&=& -1 + \kappa+ \dd\\
u(T^{01})&=& 1+ \OO(\dd)
\end{eqnarray}
\end{lemma}

Calling $G(t)=g(x(t),\al v(t),u(t))$, we have:
\begin{eqnarray*}
(- (1-\dd)+\kappa)\al &=& \int _{0}^{T^{01}} G(t)dt = T^{01}(g(x_0,0,u_0)
+\int _{0}^{T^{01}}t G'(t^* )
dt
\end{eqnarray*}
where $0\le t^*\le t$.
We use now that $g(x_0,0,u_0)=g(x_0,0,1)+ \OO(\dd)=-1+\OO(\dd)$ and that there exists a constant $L>0$ only depending
of the bounds of the functions $f$, $g$ and their derivatives in the compact $\KK$ such that
$$
|G'(t^*)|=|\frac{d}{dt}(g(x(t),\al v(t),u(t))) (t=t^*)|\le L (1+ \frac{\dd}{\kappa \al}),
$$
obtaining
$$
(- (1-\dd)+\kappa)\al=
T^{01}(-1 +\OO(\dd))+ (T^{01})^2 \OO(1+ \frac{\dd}{\kappa \al})
$$
Now, using the bounds of lemma \ref{lemma:aprioribounds} one has
$$
(- (1-\dd)+\kappa)\al=
T^{01}(-1 +\OO(\dd)+O(\al) + \OO(\frac{\dd}{\kappa}))
$$
which gives taking $|\al| \le \sigma ^*$, $0<\dd \le \sigma ^*$, $0<\frac{\dd}{\kappa} \le \sigma ^*$ for some $0<\sigma ^*\le \al_0$ small enough:
$$
T^{01}= \al +\OO(\al \kappa,\al\frac{\dd}{\kappa}, \al ^2).
$$
Once we have the asymptotics of $T^{01}$ and using that $f(x_0,0,1)=0$, we have:
\begin{eqnarray*}
x(T^{01})&=& x_0+\int_{0}^{T^{01}}f(x(t),\al v(t),u(t))dt\\&=&
x_0+T^{01}(f(x_0,0,1)+ O(\dd) )+ (T^{01})^2 O(1+\frac{\dd}{\kappa \al})\\
&=&
x_0 + \OO(\al\frac{\dd}{\kappa},\al ^2)
\end{eqnarray*}
The values of $v(T^{01})$ and $u(T^{01})$ are given by the definition of the section $\Sigma_1$.
\qed

Next step is to compute the flow from $\Sigma_1$ to $\Sigma_2$.

\begin{lemma}\label{lem:intermediate}
With the same hypotheses of Lemma \ref{lem:first} we have:

In $\Sigma_2$:
\begin{eqnarray*}
x(T^{02})-x(T^{01})&=& \OO(\al \kappa, \al\dd)\\
v(T^{02})&=& -1+\OO( \kappa , \dd)\\
u(T^{02})&=& 1-\dd-2\kappa.
\end{eqnarray*}
In $\Sigma_3$:
\begin{eqnarray*}
x(T^{03})-x(T^{02})&=& \OO(\al \kappa\log \frac{\dd}{2})\\
v(T^{03})&=&-1+ \OO(\kappa, \dd , \kappa\log \frac{\dd}{2})\\
u(T^{03}) &=& -1+\dd.
\end{eqnarray*}
In $\Sigma_4$:
\begin{eqnarray*}
x(T^{04})-x(T^{03})&=&  \OO(\al \kappa, \al\dd, \al \kappa \log {\frac{\dd}{2}})\\
v(T^{04})&=& -1+ \dd + \kappa-C \kappa \log {\frac{\dd}{2-\dd-2\kappa}}\\
u(T^{04})&=& -1+\OO(\dd).
\end{eqnarray*}
In $\Sigma_6$:
\begin{eqnarray*}
x(T^{06})-x(T^{05})&=&  \OO(\al \kappa, \al\dd)\\
v(T^{06})&=& 1+O( \kappa, \dd )\\
u(T^{06})&=& -1+\dd+2\kappa.
\end{eqnarray*}
In $\Sigma_7$:
\begin{eqnarray*}
x(T^{07})-x(T^{06})&=&  \OO(\al \kappa, \al\dd, \al \kappa \log {\frac{\dd}{2}})\\
v(T^{07})&=& 1+\OO( \kappa, \dd )\\
u(T^{07})&=& 1-\dd.
\end{eqnarray*}
In $\Sigma_8$:
\begin{eqnarray*}
x(T^{08})-x(T^{07})&=& \OO(\al \kappa, \al\dd, \al \kappa \log {\frac{\dd}{2}})\\
v(T^{08})&=& 1- \dd- \kappa+C \kappa \log {\frac{\dd}{2-\dd-2\kappa}}\\
u(T^{08})&=& 1+\OO(\dd)
\end{eqnarray*}
\end{lemma}

To obtain the bounds in $\Sigma_2$ we use that we already know by
Lemma \ref{lem:shorttimes} that $T^{12}=\OO(\al \kappa, \al\dd)$, therefore, using that
the function $f$ is bounded we have that
$x(T^{02})-x(T^{01})= \OO(T^{12})$ which gives the required bounds.
The bound for $v(T^{02})$ is just a consequence of the fact that the solution is in $R_{12}$ and therefore
$-1 + \dd + \kappa-K \kappa \le v \le -1 + \dd + \kappa$.
Finally, by definition of $\Sigma _2$ we get that $u(T^{02})=1-\dd-2\kappa$.

Once we are in $\Sigma _2$, as we know the time $T^{23}$ required by the solution to get to $\Sigma _3$ is given by \eqref{eq:t23}
an analogous reasoning gives the bounds in this region.
The bound for $v(T^{03})$ is just a consequence of the fact that the solution is in $R_{23}$ and therefore
$-1 + \dd + \kappa-K \kappa \le v \le -1+ \dd+ \kappa-C \kappa \log {\frac{\dd}{2-\dd-2\kappa}}$.
The value of $u(T^{03})$ is given by the definition of the section $\Sigma _3$.

To obtain the bounds in $\Sigma_4$ we use that we already know by
Lemma \ref{lem:shorttimes} that $T^{34}=\OO(\al \kappa, \al\dd,\al \kappa \log {\frac{\dd}{2}})$, therefore, using that
the function $f$ is bounded we have that
$x(T^{04})-x(T^{03})= \OO(T^{34})$ which gives the required bounds.
The value of $v(T^{04})$ is given by the definition of the section $\Sigma _4$.
Finally, by definition of $\Sigma _4$ we get that $u(T^{04})=1 +\OO(\dd)$.

The rest of bounds are analogous.

\qed

Next lemma gives the time and the value of the flow in the region $R_{45}$.

\begin{lemma}\label{lem:t45}
With the same hypotheses of Lemma \ref{lem:first} we have:
$$
T^{45}= \frac{2\al}{g(x_0,0,-1)}+\OO(\kappa \al , \frac{\dd} {\kappa}\al , \al ^2, \kappa \al \log {\frac{\dd}{2}}).
$$
and
\begin{eqnarray*}
&&x(T^{05})-x(T^{04})= 2\al\frac{f(x_0,0,-1)}{g(x_0,0,-1)}+
\OO(\kappa \al , \frac{\dd} {\kappa}\al , \al ^2, \kappa \al \log {\frac{\dd}{2}})\\
&&v(T^{05})= 1-\dd - \kappa \\
&&u(T^{05})= -1 +\OO(\dd) .
\end{eqnarray*}
\end{lemma}
By lemmas \ref{lem:first} and  \ref{lem:intermediate} we know that
\begin{eqnarray*}
x(T^{04})&=&x_0 + \OO(\kappa |\log\frac{\dd}{2}|\al , \frac{ \dd}{\kappa}\al, \kappa \al, \dd \al, \al ^2)=x_0+\OO(\al) \\
v(T^{04})&=& -1+ \dd + \kappa-C \kappa \log {\frac{\dd}{2-\dd-2\kappa}}\\
u(T^{04})&=& -1+\OO(\dd).
\end{eqnarray*}
We use the fundamental theorem of calculus, calling $G(t) = g(x(t),\al v(t),u(t))$
\begin{eqnarray*}
&&
(2 -2\dd -2\kappa  + C \kappa \log\frac{\dd}{2-\dd-2\kappa})\al = \int _{T^{04}}^{ T^{05}}G(t)dt\\
&&= T^{45} g(x_0 + \OO(\al ), \OO(\al ), -1+O(\dd))+
 \int _{T^{04}}^{ T^{05}} G'(t^*)(t-T^{04})dt.\\
&&= T^{45} (g(x_0 , 0 , -1) + O(\al, \dd))+
 \int _{T^{04}}^{ T^{05}}
 G'(t^*)(t-T^{04})dt.
\end{eqnarray*}
Now, using that $|G'(t^*)|\le L(1+\frac{\dd}{\kappa \al})$ and using the same procedure as in Lemma \ref{lem:first} we obtain
$$
T^{45}= \frac{2\al}{g(x_0,0,-1)}+\OO(\kappa \al , \frac{\dd} {\kappa}\al , \al ^2, \kappa \al \log {\frac{\dd}{2}}).
$$
Once we know the asymptotic for $T^{45}$ we obtain the value of $x$ by the fundamental theorem of calculus, calling $F(t) = f(x(t),\al v(t),u(t))$
\begin{eqnarray*}
&&x(T^{05})-x(T^{04})= \int _{T^{04}}^{ T^{05}}F(t)dt\\
&&= T^{45} f(x_0 + \OO(\al ), \OO(\al), -1+\OO(\dd))+
 \int _{T^{04}}^{ T^{05}}
 F'(t^*)(t-T^{04})dt.\\
&&= T^{45} (f(x_0 , 0 , -1) + \OO(\al , \dd))+
 \int _{T^{04}}^{ T^{05}}
 F'(t^*)(t-T^{04})dt.
\end{eqnarray*}
Now, using that $|F'(t^*)|\le L(1+\frac{\dd}{\kappa\al})$ and the asymptotics for $T^{45}$ we obtain the desired asymptotics for $x$.
The asymptotics of $y( T^{05})$ and $u( T^{05})$ are given by the definition of the section $\Sigma_5$.
\qed

To complete a turn around the annulus $\bf{A}$ which gives one iterate of the Poincar\'{e} map we need to compute the time from $\Sigma_8$ to $\Sigma_0$.
This is done in next lemma, whose proof is analogous to the previous one.
\begin{lemma}
With the same hypotheses of Lemma \ref{lem:first} we have:
\begin{eqnarray*}
&&T^{80}= \al+\OO(\kappa\al , \frac{ \dd} {\kappa}\al, \kappa \al \log{\frac{\dd}{2}},\al ^2)\\
&&x(T^{1})-x(T^{08}) = \OO(\al^2, \frac{\al \dd}{\kappa})\\
&&v(T^{1})=0\\
&&u(T^{1})= 1+\OO(\dd)
\end{eqnarray*}
\end{lemma}

\noindent{\textbf{Proof of Proposition \ref{prop:filalpha}:}

Putting all the lemmas together gives that the solution beginning at $(x_0,0,u_0)\in \Sigma_0$ returns to $\Sigma_0$ after a time $T^1$ satisfying:
\begin{eqnarray*}
T^1 &=&T^{01}+T^{12}+T^{23}+T^{34}+T^{45}+T^{56}+T^{67}+T^{78}+T^{80} \\
&=& 2\al +\frac{2\al}{ g(x_0,0,-1)} + \OO(\kappa\al , \frac{ \dd} {\kappa}\al, \kappa \al \log{\frac{\dd}{2}},\al ^2),
\end{eqnarray*}
and:
$$
x(T^1)=x_0 + 2\al \frac{f(x_0,0,-1)}{g(x_0,0,-1)}+\OO(\kappa\al , \frac{ \dd} {\kappa}\al, \kappa \al \log{\frac{\dd}{2}},\al ^2)
$$
which is the value of the $x$ coordinate after one iteration of the Poincar\'{e} map.

If we consider the solution of the Filipov vector field \eqref{eq:simplifiedF} at time $T^1$,
as $T^1=\OO(\alpha)$ we can Taylor expand the solution, and we obtain, using the asymptotics for $T^1$:
\begin{eqnarray*}
x_F(T^1) &=& x_0+ \frac{f(x_0,0,-1)}{1+g(x_0,0,-1)}T^1+ \OO(\al^2)\\
&=& x_0+ 2\al \frac{f(x_0,0,-1)}{g(x_0,0,-1)}+\OO(\kappa\al , \frac{ \dd} {\kappa}\al, \kappa \al \log{\frac{\dd}{2}},\al ^2)
 \end{eqnarray*}
Therefore we obtain:
$$
x_F(T^1)- x(T^1)= \OO(\kappa\al , \frac{ \dd} {\kappa}\al, \kappa \al \log{\frac{\dd}{2}},\al ^2).
$$

\textbf{Proof of Theorem \ref{thm:filalpha}}

To prove Theorem \ref{thm:filalpha} we just need to note that the time $T^1$ needed in one iteration of the Poincar\'{e} map from section $\Sigma_0$ to itself is of
order $\ord\alpha$. Therefore, proceeding as in Lemma \ref{lem:numbercycles}, one can see that we will need $\OO(\frac{1}{\al})$ iterations of the
Poincar\'{e} map to arrive to time $T$.
Consequently
$$
x_F(T)- x(T)= \OO(\kappa , \frac{ \dd} {\kappa}, \kappa \log{\frac{\dd}{2}},\al ).
$$
for $0< \kappa <\frac{1}{4}$, $|\al| \le \sigma ^*$, $0<\dd \le \sigma ^*$, $0<\frac{\dd}{\kappa} \le \sigma ^*$, $C\kappa |\log \frac{\dd}{2}|\le \frac{1}{2}$.
Renaming $\sigma^*= \alpha_0$ we get the result.

\newpage
\section{Proof of Theorem \ref{thm:utkinalpha}: slow manifold gives nonlinear sliding to $\ord{\alpha}$ for $\alpha<0<\eps$}\label{sec:uproof}

\subsection{An attracting invariant curve}

We will first show that there exists an invariant curve, $\op Q$, on which the dynamics is a perturbation of the Utkin dynamics \eref{xutkin}.
We then show that this curve is an attractor.

Writing 
\eref{bslow} with $\alpha$ (and thus $\kappa$) negative gives
\begin{eqnarray}
\label{eq:fast}
\begin{array}{rll}
 \dot x&=& f\bb{x,|\alpha| v;u}\;,\\
|\alpha| \dot v&=&g\bb{x,|\alpha| v;u}\;,\\
|\kappa||\alpha| \dot u&=&\phi\bb{\frac{v-u}{|\kappa|}}-u\;.
\end{array}
\end{eqnarray}

Taking the limit $\alpha\rightarrow0$
gives:
\begin{eqnarray}\label{bslow0}
\begin{array}{rll}
\dot x&=&f\bb{x,0;u}\;,\\
0&=&g\bb{x,0;u}\;,\\
0&=&\phi\bb{\frac{v-u}{|\kappa|}}-u\;,
\end{array}
\end{eqnarray}
which is formally similar
to Utkin's system \eref{avcon}.
This system defines a slow one-dimensional system on the critical manifold which is now a curve $\op Q$:
\begin{equation}\label{Cut}
\op Q=\cc{\;(x,v,u):\;g\bb{x,0;u}=0,\;u=\phi\bb{\mbox{$\frac{v-u}{|\kappa|}$}}\;}\;,
\end{equation}
so that on $\op Q$ the dynamics is Utkin's.

To find the dynamics outside $\op Q$ we rescale time, denoting the derivative with respect to $\tau=t/|\alpha|$ with a prime, then
\begin{eqnarray}
\label{slowv}
\begin{array}{rll}
 x'&=&|\alpha| f\bb{x,|\alpha| v;u}\;,\\
v'&=&g\bb{x,|\alpha| v;u}\;,\\
|\kappa|u'&=&\phi\bb{\frac{v-u}{|\kappa|}}-u\;.
\end{array}
\end{eqnarray}
Setting $\alpha=0$ gives a two-dimensional fast subsystem
\begin{eqnarray}\label{slowv0}
\begin{array}{rll}
 x'&=&0\;,\\
v'&=&g\bb{x,0;u}\;,\\
|\kappa|u'&=&\phi\bb{\frac{v-u}{|\kappa|}}-u\;,
\end{array}
\end{eqnarray}
whose equilibria are the set $\op Q$.

The invariant manifold geometry is illustrated in \fref{fig:utkin}.
\begin{figure}[h!]\centering\includegraphics[width=0.75\textwidth]{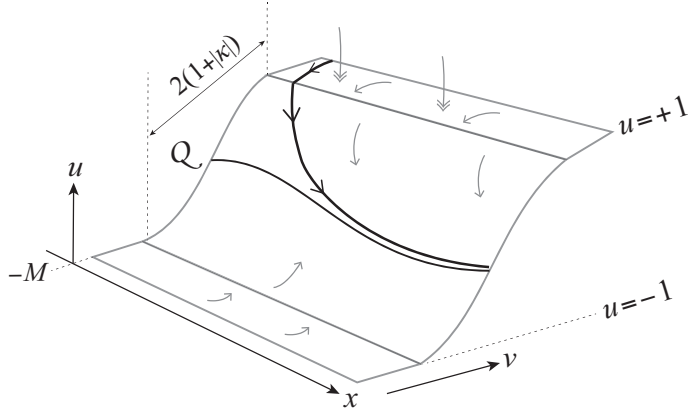}
\vspace{-0.3cm}
\caption{\footnotesize\sf The  manifold $\Phi(\frac{v-u}{|\kappa|})-u=0$, the critical curve $\op Q$, and an illustrative trajectory.}\label{fig:utkin}\end{figure}

Unlike the case $\alpha>0$ we can proceed more simply here by keeping $\kappa=\eps/\alpha$ small but nonvanishing.
By fixing $\kappa\neq0$ the function $\phi\bb{\frac{v-u}{|\kappa|}}$ will remain smooth and we can apply Fenichel's theory of normally hyperbolic slow manifolds.
Applying Fenichel's theory for two fast variables $v$ and $u$ and a slow variable $x$,
we can first show that the invariant curve $\op Q$ persists under $\alpha$-perturbation:
\begin{lemma}\label{thm:Q}
Take $\kappa < 0$. Then there exists $\al_0= \al_0(\kappa)$, such that for $\alpha_0<\alpha<0$,
then in the system \eref{bslow}
there exist attracting invariant curves $\op Q^\alpha$ which lie $\alpha$-close to $\op Q$,
on which the dynamics is differomorphic to the slow subsystem \eref{bslow0}.
\end{lemma}
\begin{proof}
The two-dimensional $(v,u)$ fast subsystem has a Jacobian derivative at $\op Q$ with determinant
\begin{equation}
\abs{\begin{array}{cc}\frac{\partial v'}{\partial v}&\frac{\partial v'}{\partial u}\\\frac{\partial u'}{\partial v}&\frac{\partial u'}{\partial u}\end{array}}
=-\frac{1}{\kappa^2}\phi'\bb{\frac{v-u}{|\kappa|}}\frac{\partial\;}{\partial u}g\bb{x,0;u}>0
\end{equation}
since the third term is negative by \eref{g} and the second is positive by \eref{phi}.
Moreover
\begin{equation}
{\rm Tr}\bb{\begin{array}{cc}\frac{\partial v'}{\partial v}&\frac{\partial v'}{\partial u}\\\frac{\partial u'}{\partial v}&\frac{\partial u'}{\partial u}\end{array}}
=-\frac1{\kappa^2}\phi'\bb{\frac{u-v}{|\kappa|}}-\frac {1}{|\kappa|}<0
\end{equation}
This implies that the curve $\op Q$ is normally hyperbolic attracting for all $x\in(-M,+M)$.
The existence of an invariant manifold $\op Q^\alpha$ in the system \eref{bslow} then follows from Fenichel's theory for a
differentiable system with one-slow and two-fast variables \cite{j95}.
\end{proof}

The next step is to show that the flow is strongly attracted towards $\op Q^\alpha$, and hence is closely approximated by the Utkin dynamics on $\op Q$.

\subsection{A positively invariant block}
By hypotheses \eqref{g} we know there exists $0<u^*<1$,  such that, reducing $\alpha_0$ if necessary,
for $|\alpha|\le \alpha_0$, $|x|\le M$, $|v|\le M$ we have that:
\begin{equation}\label{eq:gG}
\begin{array}{rcl}
 g(x,|\alpha|v,u) \le - G <0, \quad |x|\le M, \ u^* <u\le 2M \\
 g(x,|\alpha|v,u) \ge G >0, \quad |x|\le M, \ -2M \le u < -u^*
\end{array}
\end{equation}


%
\begin{figure}[h!]\centering\includegraphics[width=1\textwidth]{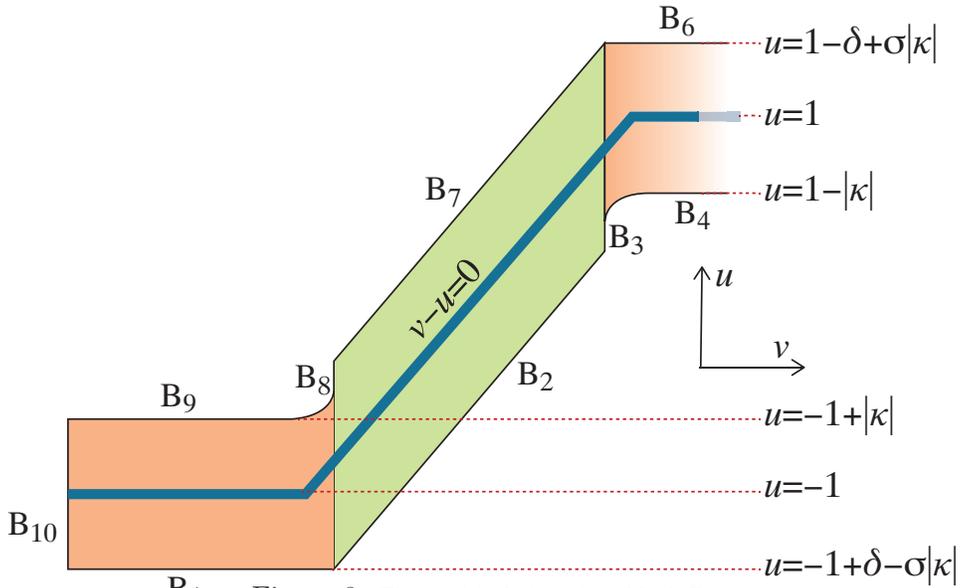}
\vspace{-1.5cm}\caption{\footnotesize\sf The positively invariant block $\bf B$.}
\label{fig:block}
\end{figure}

\begin{lemma}\label{lem:block1}
Take $\kappa <0$. There exists $0<\delta <|\kappa|$ small enough such the surface
$$
\mathbf{S}=\{(x,v,u), \ |x|< M, \ 1-\delta \le v \le M, \ \phi(\frac{v-u}{|\kappa|})-u-|\kappa|=0 \},
$$
in the region $1-\delta <v<1$ is bounded from below by the plane $u=v-|\kappa|$ and by $u=v$ from above.
\end{lemma}

\begin{proof}
We will use that for $0<1-\xi$ small enough $\phi(\xi)>\xi$.
Observe that the surface $\mathbf{S}$ contains the plane $\{(x,v,u), \ |x|< M, \ 1\le v \le M, \ u=1-|\kappa| \}$.
The intersection of $\mathbf{S}$ with $u=v-|\kappa|$ is the line $u=1-|\kappa|,\ v=1$. Observe that at this point $\frac{v-u}{|\kappa|}=1$, therefore choosing
$\delta$ small enough and $1-\delta \le v\le 1$ we can ensure that if $(x,v,u)$ are in $\mathbf{S}$ then,
$u=\phi(\frac{v-u}{|\kappa|})-|\kappa|>\frac{v-u}{|\kappa|}-|\kappa|$.
On the surface $u=\frac{v-u}{|\kappa|}-|\kappa|$ one has that
$u= \frac{v}{1+|\kappa|}- \frac{|\kappa|^2}{1+|\kappa|}>v-|\kappa|$ if $v<1$, therefore $u=v-|\kappa|$
bounds the surface $\mathbf{S}$ from bellow in the region $1-\delta <v<1$.

For the upper bound we just use that $u\le 1-|\kappa|$ and that $\delta <|\kappa|$.

\end{proof}
Lets call $u_\delta$ the $u$-coordinate of the intersection of $\mathbf{S}$ with $v=1-\delta$, that is:
$$
\phi(\frac{1-\delta-u_\delta}{|\kappa|})-u_\delta-|\kappa|=0
$$
we know that $1-\delta -|\kappa| <u_\delta < 1-\delta$.

Take $\sigma>0$ such that $\delta <\sigma \kappa$ (we will fix its value in Proposition \ref{prop:blockutkin}) and consider now the block $\mathbf{B}$
whose exterior borders are given by the sets $\{ x=\pm M\}$ and:
\begin{eqnarray*}
\mathbf{B_1}&=& \{(x,v,u), \ |x|< M, \ -M \le v \le -1+\delta ,\; u= -1+\delta-\sigma |\kappa|\}\\
\mathbf{B_2}&=& \{ (x,v,u), \ |x|< M, -1+\delta \le v \le 1-\delta,\; u= v-\sigma |\kappa| \}\\
\mathbf{B_3}&=& \{(x,v,u), \ |x|< M, \ v= 1-\delta ,\; 1-\delta-\sigma|\kappa|\le u \le u_\delta \}\\
\mathbf{B_4}&=& \{(x,v,u), \ |x|< M, \ 1-\delta \le v \le M, \; \phi(\frac{v-u}{|\kappa|})-u-|\kappa|=0 \}\\
\mathbf{B_5}&=& \{(x,v,u), \ |x|< M, \ v = M ,\; 1-|\kappa| \le u\le  1-\delta+\sigma |\kappa|\}\\
\mathbf{B_6}&=& \{(x,v,u), \ |x|< M, \ 1-\delta \le v \le M, \; u= 1-\delta+\sigma |\kappa| \}\\
\mathbf{B_7}&=&\{ (x,v,u), \ |x|\le M, -1+\delta \le v \le 1-\delta, \;u= v+\sigma |\kappa| \}\\
\mathbf{B_8}&=& \{(x,v,u), \ |x|< M, \ v= -1+\delta , \;-u_\delta \le u \le -1+\delta+\sigma\kappa \}\\
\mathbf{B_9}&=& \{(x,v,u), \ |x|< M, \ -M \le v \le -1+\delta, \; \phi(\frac{v-u}{|\kappa|})-u-|\kappa|=0 \}\\
\mathbf{B_{10}}&=& \{(x,v,u), \ |x|< M, \ v =- M, \; -1+\delta-\sigma|\kappa| \le u\le -1+|\kappa|\}
 \end{eqnarray*}

Observe that, being $\mathbf{B_4}=\mathbf{S}$ (and analogously for $\mathbf{B_9}$), we know that, by Lemma \ref{lem:block1}, $\mathbf{B}$ is well defined.

We will see that the solutions of system \eqref{slowv} which enter this block can only leave it through $|x|=M$.
\begin{proposition}\label{prop:blockutkin}
 Let $\sigma = G+2$ (see \eqref{eq:gG}).
 Then for $|\alpha| \le \alpha _0$, 
 $0< \delta \le \kappa \le \frac{1-u^*}{2\sigma}$,
any solution of system \eqref{slowv} entering $\mathbf{B}$ leaves it through the boundaries $x=\pm M$.
\end{proposition}

\begin{proof}

$\bullet$ In $\mathbf{B_1}$, as $u= -1+\delta-\sigma |\kappa|<-1$, $u'= \frac{1}{|\kappa|}\left[\phi\bb{\frac{v-u}{|\kappa|}}-u\right] >0$,
therefore the flow points inwards $\textbf{B}$ along this border.
Analogously in $\mathbf{B_6}$.

$\bullet$ In $\mathbf{B_2}$ $u= v-\sigma|\kappa|$ and the exterior normal vector is $(0, 1,-1)$, therefore,
the condition to ensure that the vector field points inwards is:
$$
g(x,|\alpha|v,u) - \frac{1}{|\kappa|}\left[\phi(\frac{v-u}{|\kappa|})-u\right]<0, \ -1+\delta\le v \le 1-\delta
$$
which gives, using that $g$ satisfies \eqref{eq:gG}, $u= v-\sigma|\kappa|$ and that $\sigma=G+2$:
$$
g(x,|\alpha|v,u) - \frac{1}{|\kappa|}[\phi(\sigma)-v+\sigma|\kappa|] < G -\sigma - \frac{1}{|\kappa|}[1-v]\le -2-\frac{\delta}{\kappa}<0
$$
Analogously for $\mathbf{B_7}$.

$\bullet$ In $\mathbf{B_3}$,
$u \ge 1-\delta -\sigma |\kappa|\ge 1-2\sigma |\kappa| >u^*$ and therefore, by \eqref{eq:gG},
$g(x,v,u)<-G$, which implies that the flow points inwards
$ \mathbf{B}$ at this border. Analogously for $\mathbf{B_8}$.

$\bullet$ In $\mathbf{B_4}$, the exterior normal vector is $(0, \frac{1}{|\kappa|}\phi'(\frac{v-u}{|\kappa|}), -\frac{1}{|\kappa|}\phi'(\frac{v-u}{|\kappa|})-1)$
Therefore, we need to see that:
$$
g(x,|\alpha|v, u)\frac{1}{|\kappa|}\phi'(\frac{v-u}{|\kappa|})- \left(\frac{1}{|\kappa|}\phi'(\frac{v-u}{|\kappa|})+1\right)
\frac{1}{|\kappa|}\left(\phi(\frac{v-u}{|\kappa|})-u\right)<0
$$
that, for points in $\mathbf{B_4}$ gives:
$$
g(x,|\alpha|v, u)\frac{1}{|\kappa|}\phi'(\frac{v-u}{|\kappa|})- \left(\frac{1}{|\kappa|}\phi'(\frac{v-u}{|\kappa|})+1\right)<0
$$
The only observation is that in $\mathbf{B_4}$, $u\ge u_\delta\ge 1-\delta -\sigma |\kappa|>u^*$.
therefore, by \eqref{eq:gG}
we know that $g<0$.
Analogously for $\mathbf{B_9}$.

$\bullet$ In $\mathbf{B_5}$, $u \ge u^* $ and therefore $g<0$. Analogously for $\mathbf{B_{10}}$.

\end{proof}

\begin{lemma}\label{thm:Qat}
With the same hypotheses of Proposition \ref{prop:blockutkin},
take initial conditions $z_0=(x_0,v_0,u_0)$ in the interior of $\mathbf{B}$ with $x_0=x_U(0)$, where $x_U(t)$ is the solution of \eref{xutkin}.
Then, For  $t\in[0,T]$
we have $|x(t)-x_U(t)|<\ord\alpha$.
\end{lemma}

\begin{proof}
We first show that the orbit of the point $z_0$
is attracted to the invariant curve $\op Q^\al$ given by the Fenichel theorem and which is $\alpha$ close to $\op Q$ (see \eqref{Cut}). 

We already know, by Fenichel Theorem, that $\op Q ^\al$ is locally attracting.
Due to Proposition \ref{prop:blockutkin}, in fact the entire flow in the region $\mathbf{B}$ considered is
attracted to $\op Q^\al$.

Analogously to Lemma \ref{lem:cotainftemps}, the time needed by the solution $z(t)$ such that $z(0)=z_0$ to reach $x=\pm M$ is of order $\OO(1)$, but the time $t_1$ needed to reach
the neighbourhood of attraction (which is of order $1$) of $\op Q^\al$ is of order $\OO(\al)$, consequently $x(t_1)-x_0=\OO(\al)$ and,
using that $x_0=x_U(0)$, we obtain that $x(t)-X_U(t)=\OO(\al)$ for $0\le t\le t_1$.
To establish that the resulting dynamics is approximated by \eref{bslow0} for $0\le t\le T$
we then need to look more closely at the expression of $\op Q$ and hence of $\op Q^\alpha$.
Firstly,
let us observe that we have that $\op Q$ lies between $\mathbf{B_2}$ and $\mathbf{B_7}$ and therefore in $|v|<1$.
Within this region $\op Q$ lies on the mid-branch of the surface $u=\phi\bb{\frac{v-u}{|\kappa|}}$.
The definition \eref{phi} of $\phi$ implies that the middle branch lies in $|v-u|<|\kappa|$, which tends to $v=u$ as $\kappa\rightarrow0$,
so for small $\kappa<0$ the branch is given by $v=u+|\kappa|\phi^{-1}(u)$.
Then $\op Q$ in the limit $\alpha=0$ is the solution of
$$\begin{array}{rll}
0&=&g\bb{x,0;u}\;,\\
0&=&\phi\bb{\frac{v-u}{|\kappa|}}-u\;.
\end{array}$$
which, calling $h(x)$ the function such that $g(x,0,h(x)=0$, for $|x|\le M$, is given by
$$
\op Q=\cc{(x,u,v)\in(-M,+M)\times\mathbb R^2\;:\;u=h(x),\;v=u+|\kappa|\phi^{-1}(u)}\;.
$$
Applying Fenichel theory for $\alpha<0$, the invariant manifold $\op Q^\alpha$ is a regular $\alpha$-perturbation of $\op Q$,
$$
\op Q^\alpha=\cc{(x,u,v)\in(-M,+M)\times\mathbb R^2\;:\;\begin{array}{l}u=h(x)+\ord{\alpha},\\v=u+|\kappa|\phi^{-1}(u)+\ord{\alpha}\end{array}}\;.
$$
Finally, on $\op Q^\alpha$ the system is an $\alpha$ perturbation of that on $\op Q$, and so for $0\le|\alpha|<|\kappa|\ll1$ we have
$$
\begin{array}{rll}
\dot x&=&f(x,0;u)+\ord{\alpha}\;,\\
0&=&g\bb{x,0;u}+\ord{\alpha}\;,
\end{array}$$
the solution  $z(t)$ stays in the neighbourhood of $\op Q^\al$ for $t\in(0,T)$ and therefore
$$
x(t)=x_U(t)+\ord{\alpha}\;.
$$
\end{proof}

Note here that we keep $\kappa$ non-vanishing, whereas for the hysteretic cases we let $\kappa\rightarrow0$ choosing, for instance, $\kappa=\al$ in
Corollary \ref{cor:filalpha}.
Keeping $\kappa$ bounded away from zero here allows us to keep the righthand side of the ordinary differential equations smooth,
in particular to keep $\phi$ smooth (avoiding $\phi(w/|\kappa|)$ becoming a step function for $\kappa\rightarrow0$).
This allows us to apply Fenichel's theory directly.
Since the outcome of the theorem already gives an $\alpha$-perturbation of the Utkin dynamics when we consider small
$\alpha<0$, the $\kappa\neq0$ result is sufficient here.

\newpage

\bigskip
\noindent{\bf Acknowledgements. }MRJ's contribution carried out primarily at UPC courtesy of UPC and DANCE, and with support from EPSRC Grant Ref: EP/J001317/2. C. Bonet and T.M-Seara are been partially supported by the Spanish MINECO-FEDER Grants MTM2015-65715-P and the Catalan Grant 2014SGR504. Tere M-Seara is also supported by the Russian Scientific Foundation grant 14-41-00044 and the European Marie Curie Action
FP7-PEOPLE-2012-IRSES: BREUDS.


\begin{thebibliography}{10}

\bibitem{aizerman12}
M.~A. Aizerman and E.~S. Pyatnitskii.
\newblock Fundamentals of the theory of discontinuous systems {I,II}.
\newblock {\em Automation and Remote Control}, 35:1066--79, 1242--92, 1974.

\bibitem{and59}
A.~A. Andronov, A.~A. Vitt, and S.~E. Khaikin.
\newblock {\em Theory of oscillations}.
\newblock Moscow: Fizmatgiz (in Russian), 1959.

\bibitem{bc08}
M.~di~Bernardo, C.~J. Budd, A.~R. Champneys, and P.~Kowalczyk.
\newblock {\em Piecewise-Smooth Dynamical Systems: Theory and Applications}.
\newblock Springer, 2008.

\bibitem{dieci2011}
L.~Dieci and L.~Lopez.
\newblock Sliding motion on discontinuity surfaces of high co-dimension. a
  construction for selecting a {F}ilippov vector field.
\newblock {\em Numer. Math.}, 117:779--811, 2011.

\bibitem{f64}
A.~F. Filippov.
\newblock Differential equations with discontinuous right-hand side.
\newblock {\em American Mathematical Society Translations, Series 2},
  42:19--231, 1964.

\bibitem{f88}
A.~F. Filippov.
\newblock {\em Differential Equations with Discontinuous Righthand Sides}.
\newblock Kluwer Academic Publ. Dortrecht, 1988.

\bibitem{fluggelotz}
I.~Fl\"ugge-Lotz.
\newblock {\em Discontinuous Automatic Control}.
\newblock Princeton University Press, 1953.

\bibitem{tsg11}
M.~Guardia, T.~M. Seara, and M.~A. Teixeira.
\newblock Generic bifurcations of low codimension of planar {F}ilippov systems.
\newblock {\em J. Differ. Equ.}, pages 1967--2023, 2011.

\bibitem{hajek1}
O.~H\'ajek.
\newblock Discontinuous differential equations, {I}.
\newblock {\em J. Differential Equations}, 32(2):149--170, 1979.

\bibitem{hermes68}
H.~Hermes.
\newblock Discontinuous vector fields and feedback control.
\newblock {\em Differential Equations and Dynamical Systems}, pages 155--165,
  1967.

\bibitem{j13iso}
M.~R. Jeffrey.
\newblock Dynamics at a switching intersection: hierarchy, isonomy, and
  multiple-sliding.
\newblock {\em SIADS}, 13(3):1082--1105, 2014.

\bibitem{j13error}
M.~R. Jeffrey.
\newblock Hidden dynamics in models of discontinuity and switching.
\newblock {\em Physica D}, 273-274:34--45, 2014.

\bibitem{j16jitter}
M.~R. Jeffrey, G.~Kafanas, and D.~J.~W. Simpson.
\newblock Jitter in dynamical systems with intersecting discontinuity surfaces.
\newblock {\em submitted}, 2016.

\bibitem{j95}
C.~K. R.~T. Jones.
\newblock {\em Geometric singular perturbation theory}, volume 1609 of {\em
  Lecture Notes in Math. pp. 44-120}.
\newblock Springer-Verlag (New York), 1995.

\bibitem{krg03}
Yu.~A. Kuznetsov, S.~Rinaldi, and A.~Gragnani.
\newblock One-parameter bifurcations in planar {F}ilippov systems.
\newblock {\em Int. J. Bif. Chaos}, 13:2157--2188, 2003.

\bibitem{lindelof}
E.~Lindel\"of.
\newblock Sur l'application de la m\'ethode des approximations successives aux
  \'equations diff\'erentielles ordinaires du premier ordre.
\newblock {\em C. R. Hebd. Seances Acad. Sci}, 114:454--457, 1894.

\bibitem{BonetS16}
Carles~Bonet Rev{\'e}s and Tere~M. Seara.
\newblock Regularization of sliding global bifurcations derived from the local
  fold singularity of filippov systems.
\newblock 02 2014.

\bibitem{slot91}
{J-J}.~E. Slotine and W.~Li.
\newblock {\em Applied Nonlinear Control}.
\newblock Prentice Hall, 1991.

\bibitem{t07}
M.~A. Teixeira, J.~Llibre, and P.~R. da~Silva.
\newblock Regularization of discontinuous vector fields on {$R^3$} via singular
  perturbation.
\newblock {\em Journal of Dynamics and Differential Equations}, 19(2):309--331,
  2007.

\bibitem{u77}
V.~I. Utkin.
\newblock Variable structure systems with sliding modes.
\newblock {\em IEEE Trans. Automat. Contr.}, 22, 1977.

\bibitem{u74}
V.~I. Utkin.
\newblock {\em Sliding modes and their application in variable structure
  systems}, volume (Translated from the Russian).
\newblock MiR, 1978.

\bibitem{u92}
V.~I. Utkin.
\newblock {\em Sliding modes in control and optimization}.
\newblock Springer-Verlag, 1992.

\end{thebibliography}

\end{document}